\documentclass[12pt]{amsart}
\usepackage{amsmath,amsthm,amsfonts,amssymb,eucal, hyperref}
\usepackage{float}
\usepackage{tikz}
\usepackage{tikzit}
\tikzstyle{Vertex}=[fill=white, draw=black, shape=circle]
\tikzstyle{base vertex}=[fill=white, draw=black, shape=diamond]

\tikzstyle{Arrow edge}=[->]
\usepackage{pdfpages}

\newcommand{\oq}{\ {\raise 7pt\hbox{${\scriptstyle\circ}$}}
\kern -7pt{%\lower 2pt
\hbox{$Q$}}}

\newcommand{\R}{ \mathbb R}
\newcommand{\Q}{ \mathbb Q}

\newcommand {\ba}{\mathbf a}

\newcommand {\bk}{\mathbf k}

\newcommand {\bm}{\mathbf m}

\newcommand {\bze}{\mathbf 0}

\newcommand {\bn}{\mathbf n}

\newcommand{\CL}{\mathcal L}

\newcommand{\CP}{\mathcal P}
\newcommand{\CQ}{\mathcal Q}

\DeclareMathOperator{\range}{{range}}

\newcommand{\1}
{{\,\vrule depth3pt height9pt}{\vrule depth3pt height9pt}
{\vrule depth3pt height9pt}{\vrule depth3pt height9pt}\,}

\DeclareMathOperator {\dist} {{dist}}

\DeclareMathOperator{\reg}{{reg}}

\DeclareMathOperator{\safe}{{safe}}
\newtheorem{thm}{Theorem}[section]
\newtheorem{cor}[thm]{Corollary}
\newtheorem{cla}[thm]{Claim}
\newtheorem{lem}[thm]{Lemma}
\newtheorem{prop}[thm]{Proposition}

\theoremstyle{definition}
\newtheorem{defn}[thm]{Definition}%[section]

\newtheorem*{remark}{Remark}
\newtheorem{rem}[thm]{Remark}

\numberwithin{equation}{section}

\newcommand{\bee}{\begin{equation}}
\newcommand{\ene}{\end{equation}}
\newcommand{\bees}{\begin{equation*}}
\newcommand{\enes}{\end{equation*}}
\newcommand{\bes}{\begin{split}}
\newcommand{\ens}{\end{split}}

\newcommand{\bet}{\begin{thm}}
\newcommand{\ent}{\end{thm}}
\newcommand{\bel}{\begin{lem}}
\newcommand{\enl}{\end{lem}}
\newcommand{\bec}{\begin{cor}}
\newcommand{\enc}{\end{cor}}
\newcommand{\becl}{\begin{cla}}
\newcommand{\encl}{\end{cla}}
\newcommand{\bep}{\begin{proof}}
\newcommand{\enp}{\end{proof}}
\newcommand{\ber}{\begin{rem}}
\newcommand{\enr}{\end{rem}}
\newcommand{\ep}{\varepsilon}
\newcommand{\la}{\lambda}

\newcommand{\Z}{\mathbb Z}

\makeatletter
\def\square{\RIfM@\bgroup\else$\bgroup\aftergroup$\fi
  \vcenter{\hrule\hbox{\vrule\@height.6em\kern.6em\vrule}\hrule}\egroup}
\makeatother

 \usepackage{hyperref}
\def\level{\mathtt{level}}
\def\dio{\mathrm{dio}}

\def\maxlevel{\mathtt{maxlevel}}
\def\denominators{\mathtt{den}}

\def\den{\mathtt{den}}
\def\mbe{M_{\beta}}
\def\totallevel{\mathtt{totallevel}}
\def\cont{\mathtt{Cont}}

\def\range{\mathtt{Range}}
\def\loops{\mathtt{loops}}
\def\nblevel{\mathtt{nblevel}}

\def\height{\mathtt{height}}

\def\nbloops{\mathtt{nbloops}}
\def\downedges{\mathtt{downedges}}
\def\singdownedges{\mathtt{singdownedges}}
\def\singden{\mathtt{singden}}

\def\height{\mathtt{height}}
\def\safedist{\mathtt{safedist}}
\usepackage[margin=1in]{geometry}

\begin{document}

\title[Perturbation theory]
{Perturbation theory}
\title[Convergence of perturbation theory]
{Convergence of perturbation series for unbounded monotone quasiperiodic operators}
\author[I. Kachkovskiy]
{Ilya Kachkovskiy}
\address{Department of Mathematics\\ Michigan State University\\
Wells Hall, 619 Red Cedar Rd\\ East Lansing, MI\\ 48824\\ USA}
\email{ikachkov@msu.edu}
\author[L. Parnovski]
{Leonid Parnovski}
\address{Department of Mathematics\\ University College London\\
Gower Street\\ London\\ WC1E 6BT\\ UK}
\email{leonid@math.ucl.ac.uk}
\author[R. Shterenberg]
{Roman Shterenberg}
\address{Department of Mathematics\\ University of Alabama, Birminghan\\
Campbell Hall\\1300 University Blvd\\ Birmingham, AL\\ 35294\\USA }
\email{shterenb@math.uab.edu}

%\author[L. Parnovski]
%{Leonid Parnovski}
%\address{Department of Mathematics\\ University College London\\
%Gower Street\\ London\\ WC1E 6BT\\ UK}
%\email{Leonid@math.ucl.ac.uk}

%\keywords{Periodic operators}
%\subjclass[2000]{Primary 35P20, 47G30, 47A55; Secondary 81Q10}
%\thanks{This work is supported by the Royal Society}
%\copyrightinfo{2002}{Alexander V. Sobolev}

\date{\today}

%\begin{abstract}
%\input polyabstract.tex
%\end{abstract}
%\begin{abstract}
%We consider a periodic pseudodifferential operator $H=(-\Delta)^l+B$ ($l>0$) in $\R^d$ which satisfies the
%following conditions: (i) the symbol of $H$ is smooth in $x$, and (ii) the perturbation
%$B$ has order smaller than $2l$. Under these assumptions, we prove that the spectrum
%of $H$ contains a half-line.

%\noindent {{\bf AMS Subject Classification:} 35P20 (58J40, 58J50, 35J10)}
%\nl

%{{\bf Keywords:} Bethe-Sommerfeld conjecture, periodic problems, pseudo-differential operators,
%spectral gaps}
%\end{abstract}
\begin{abstract}
We consider a class of unbounded quasiperiodic Schr\"odinger-type operators on $\ell^2(\Z^d)$ with monotone potentials (akin to the Maryland model) and show that the Rayleigh--Schr\"odinger perturbation series for these operators converges in the regime of small kinetic energies, uniformly in the spectrum. As a consequence, we obtain a new proof of Anderson localization in a more general than before class of such operators, with explicit convergent series expansions for eigenvalues and eigenvectors. This result can be restricted to an energy window if the potential is only locally monotone and one-to-one. A modification of this approach also allows the potential to be non-strictly monotone and have a flat segment, under additional restrictions on the frequencies.
\end{abstract}
\maketitle
\section{Introduction}
In this paper, we consider a class of quasiperiodic Schr\"odinger-type operators on $\ell^2(\Z^d)$
\bee
\label{eq_h_def}
(H(x_0)\psi)_{\bn}=\varepsilon\sum_{\bm\in \Z^d}\varphi_{\bn-\bm}\psi_{\bm}+f(x_0+\bn\cdot\omega)\psi_{\bn}.
\ene
The frequency vector $\omega\in (-1/2,1/2)^d$ will always satisfy the assumption that $1,\omega_1,\ldots,\omega_d$ are rationally independent. The parameter $x_0\in \R\setminus(\Z+1/2+\Z^d\cdot\omega)$ is the quasiperiodic phase. We will always assume $\varphi\in\ell^1(\Z^d)$, so that the first term in the right hand side (usually referred to as hopping or kinetic term) is a bounded operator on $\ell^2(\Z^d)$. Most of the time we will assume that $\varphi_{\bn}\neq 0$ only for finitely many $\bn\in \Z^d$. The most common example is the discrete nearest-neighbor Laplace operator:
\bee
\label{eq_discrete_laplacian}
\varphi=\sum\limits_{\bn\colon |\bn|_1=1} e_{\bn},
\ene
where $\{e_{\bn}\colon \bn\in\Z^d\}$ is the standard basis in $\ell^2(\Z^d)$ and $|\cdot|_1$ denotes the $\ell^1$-norm. 

The real-valued function $f$ is initially defined on $(-1/2,1/2)$; we assume it to be continuous and satisfy
$$
f(-1/2+0)=-\infty,\quad f(1/2-0)=+\infty.
$$
We then extend $f$ periodically into $\R\setminus(\Z+1/2)$. The assumptions on $x_0$ and $\omega$ imply that \eqref{eq_h_def} is a well-defined unbounded self-adjoint operator on $\ell^2(\Z^d)$. In the present work, the most important assumption will be that $f$ is non-decreasing on $(-1/2,1/2)$. In most of the paper, we assume that $f$ is strictly increasing, with the derivative bounded from below, but in the last section $f$ will be allowed to have flat pieces. 

Spectral theory of operators \eqref{eq_h_def} with strictly monotone functions $f$ dates back to the classical Maryland model \cite{GFP,Simon_maryland,FP}, where $f(x)=\tan(\pi x)$ and the kinetic term is the discrete Laplacian \eqref{eq_discrete_laplacian}. Similarly to many other quasiperiodic models, the spectral types of such operators depend on the arithmetic properties of the frequencies. We will call the frequency vector $\omega$ {\it Diophantine} if, for some $C_{\dio}>0$ and $\tau>d+1$ we have
\bee
\label{eq_diophantine_definition}
\|\bn \cdot\omega\|:=\dist(\bn\cdot\omega,\Z)\ge C_{\mathrm{dio}}|\bn|^{-\tau},\quad \forall \bn\in \Z^d\setminus\{\bze\}.
\ene
It is known that for all $\varepsilon>0$, all Diophantine $\omega$, and all values of $x_0$, the Maryland model has Anderson localization, that is, dense pure point spectrum and exponentially decaying eigenfunctions. In $d=1$, a complete description of the spectral type (in particular, the transitions between pure point and singular continuous spectra depending on whether the frequency is closer to a Diophantine or a Liouville irrational number) was obtained in in \cite{Wencai}. See also \cite{JY2} for an alternative proof in the Diophantine setting.
The case of more general $f$ was studied in \cite{Bellissard}, and it was shown that the operator has Anderson localization for Diophantine frequencies and $0<\varepsilon<\varepsilon_0(\omega,f,\varphi)$. The function $f$ was assumed to be strictly monotone and to have a meromorphic continuation into a strip in $\mathbb C$. 
The approach of \cite{Bellissard} is based on a KAM-type scheme: a diagonalization of $H(x_0)$ is constructed via an infinite sequence of unitary transformations. In a recent work \cite{Ilya}, Anderson localization for $d=1$ is shown for all $\varepsilon>0$ and all Diophantine frequencies, under the assumption that $\log|f|\in L^1(-1/2,1/2)$. However, the non-perturbative method of \cite{Ilya}, based on \cite{JK}, cannot be extended to $d>1$. We also mention \cite{JY}, where singular continuous spectrum of operators with meromorphic  potentials is studied.

In the present paper, we study Anderson localization for \eqref{eq_h_def} in the same perturbative setting as \cite{Bellissard}, that is, for $|\varepsilon|<\varepsilon_0(\omega,f,\varphi)$. Our conditions on $f$ include all meromorphic functions from \cite{Bellissard}, but are formulated completely in terms of the first derivative of $f$. We believe that the most interesting aspect of our paper is the method: we construct explicit series for eigenvalues and eigenfunctions using the standard perturbation theory, and, to our surprise, in both cases {\it the Rayleigh-Schr\"odinger perturbation series converge} as is, without the need of multiple KAM-type steps. As a consequence, we are able to write down complete representations of eigenvalues and eigenvectors in terms of converging power series in $\varepsilon$. Compared to the previous work, our method also covers two new cases (the summary of the main results is provided in the next section):
\begin{itemize}
	\item The case when $f$ is monotone and one-to-one on an interval $(a,b)\subset (-1/2,1/2)$, with some regularity properties outside of $(a,b)$ (including $f^{-1}(f(a,b))=(a,b)$), but does not have to be monotone outside of $(a,b)$. In this case, we obtain localization on the energy interval slightly smaller than $f(a,b)$, see Theorem \ref{th_local_localization}.
	\item The case when $f$ has a flat segment, but is Lipschitz monotone outside of that segment. In this case, under some additional assumptions on the length of the segment and the frequencies, we can obtain complete Anderson localization, see Theorem \ref{th_main_flat}.
\end{itemize}

To our best knowledge, this is the first class of examples of convergent perturbation series in the context of Schr\"odinger operators with dense eigenvalues. A related phenomenon for the classical KAM was observed in \cite{Eliasson} in the context of ODEs (which does not appear to be related to monotonicity). Our proof is based on careful observation of cancellations between terms of the perturbation series with the same power of $\varepsilon$ (see Remarks \ref{rem_summary1}, \ref{rem_summary2}, \ref{rem_summary3} for the outline of the cancellation mechanism). While the convergence results use the structure \eqref{eq_h_def}, the actual combinatorial procedure of grouping terms in order to prepare them for cancelling can be formulated in a more abstract context of arbitrary lattice Schr\"odinger-type operators. We believe that this procedure may be of independent interest. For example, for bounded monotone discontinuous potentials such as $f(x)=\{x\}$ the perturbation series will converge on a large (but not full measure) subset of phases $x_0$ and thus produce a large number of exponentially decaying eigenfunctions, whose contribution to the integrated density of states will approach $1$ as $\varepsilon\to 0$. However, the remaining ``resonant'' phases will require further analysis, possibly by different methods. We intend to explore appropriate modifications of our method in order to study localization and spectral gaps of these operators in subsequent publications. One can also draw some parallels between our procedure and renormalization techniques used in statistical physics, such as the work \cite{Mastroprietro} where Anderson localization in terms of two-particle functions was studied for the quasiperiodic Holstein model.

\vskip 1mm

\subsection*{Acknowledgments} The authors are grateful to David Damanik and Leonid Pastur  for useful discussions, and to the anonymous referee whose suggestions, hopefully, made the text substantially more readable. The research of LP was partially supported by EPSRC grants EP/J016829/1 and EP/P024793/1. RS was partially supported by NSF grant DMS--1814664. IK was partially supported by NSF grant DMS--1846114.

%In Section 2, we work with abstract (determinitstic) Schr\"odinger operators on $\ell^2(\Z^d)$ and construct the formal Rayleigh--Schr\"odinger perturbation series, interpreting its terms as paths on a specially constructed graph. Starting from Section 2.4, we set up the abstract cancellation procedure, which is motivated by the quasiperiodic case but may be useful in a wider class of examples. In Section 3, we start applying the abstract scheme to our operator \eqref{eq_h_def} and summarize the assumptions on the function $f$ in order for the scheme to work. In Section 4, we establish the main estimate (Theorem \ref{th_loop_stack}) which implies Anderson localization in the strictly non-decreasing case. In Section 5, we generalize the results of the previous sections to the infinite range case, where the hopping term is also a series in $\varepsilon$. We also consider the case of variable coefficients in the hopping term. In Section 6, we discuss the case of a potential with a flat segment. We verify that, under several assumptions, it can be transformed to an operator for which the results of Section 5 apply.

%The main results of the paper are Theorem \ref{main} and Corollary \ref{cor_psi}. 
\section{Rayleigh--Schr\"odinger perturbation series and main results}
\subsection{The perturbation series}
Let $V$ be a self-adjoint multiplication operator on $\ell^2(\Z^d)$, not necessarily bounded:
\bee
\label{eq_v_def}
(Vu)_\bn=V_{\bn} u_{\bn},\quad u=\sum\limits_{\bn\in \Z^d} u_{\bn} e_{\bn},
\ene
where $\{e_{\bn}:\bn\in \Z^d\}$ is the standard basis in $\ell^2(\Z^d)$. Let also $\Phi$ be a T\"oplitz-type operator: for a sequence $\{\varphi_{\bn}\}_{\bn\in \Z^d}$, we define 
\bee
\label{eq_phi_def}
(\Phi u)_{\bn}=\sum\limits_{\bm\in\Z^d}\varphi_{\bn-\bm}u_{\bm}.
\ene
We will assume that $\Phi$ is self-adjoint:
\bee
\label{eq_phi_sa_def}
\Phi_{\bm\bn}=\overline{\Phi_{\bn\bm}},\,\,\text{ that is, }\,\,\varphi_{-\bn}=\overline{\varphi_{\bn}}.
\ene
In Sections 2 -- 5, we also require that $\varphi_{\bn}\neq 0$ only for {\it finitely many} lattice points $\bn\in \Z^d$, and $\varphi_{\bze}=0$. In Section 6, this will be generalized to an infinite range case in the quasiperiodic setting. Define a family of operators
\bee
H=V+\ep \Phi,
\ene
parametrized by $\varepsilon>0$. We will not emphasize the dependence on $\varepsilon$ in the expression for $H$, as long as it is clear from the context. Fix some $\bn_0\in \Z^d$ and assume that 
\bee
\label{eq_v_nonresonant}
\quad V_{\bn}\neq V_{\bn_0}\text{ for }\bn\neq \bn_0.
\ene
In other words, the potential attains the value $V_{\bn_0}$ only at the lattice point $\bn_0$. We will call this the {\it algebraic non-resonant condition at $\bn_0$}.

If $\varepsilon=0$, then the operator $H$ has an eigenvalue $V_{\bn_0}$ with an eigenvector $e_{\bn_0}$. {\it Rayleigh--Schr\"odinger series} is a formal perturbation-theoretic expansion, which represents how that eigenvalue changes after the perturbation with $\varepsilon>0$. In order to construct these series, consider the eigenvalue equation
\bee
\label{eq_rs_eigenequation}
H\psi=\lambda\psi,
\ene
where both $\lambda$ and $\psi$ are formal series in $\ep$:
\bee
\label{eq_rs_lambda}
\lambda=\la_0+\ep\la_1+\ep^2\lambda_2+...,
\ene
\bee
\label{eq_rs_psi}
\psi=\psi_0+\ep\psi_1+\ep^2\psi_2+...
\ene
with $\la_0=V_{\bn_0}$ being the unperturbed eigenvalue and $\psi_0=e_{\bn_0}$ the unperturbed eigenfunction. Then, one considers \eqref{eq_rs_eigenequation} as an equality of formal series
\bee
\label{def_rs_series}
(V+\ep \Phi)(\psi_0+\ep\psi_1+...)=(\la_0+\ep\la_1+...)(\psi_0+\ep\psi_1+...);
\ene
in other words, an infinite system of equations obtained by comparing the coefficients at each power of $\ep$ (note that each equation will have finitely many terms). We will impose the following normalization conditions:
\bee
\psi_j\perp \psi_0,\ \ j>0.
\ene
In other words, all the correction terms are orthogonal to $e_{\bn_0}$.

It is well known and easy to see that, under the above assumptions on $V$ and $\Phi$, both $\lambda_j$ and $\psi_j$ can be uniquely determined from equating the coefficients in \eqref{def_rs_series} at each power of $\varepsilon$; see also Theorem \ref{th_lak} below. One can start these expansions at any lattice site $\bn_0$ as long as it satisfies the algebraic non-resonance condition.

\subsection{Summary of the main results} The results of the present paper rely on certain regularity conditions on the function $f$, which can be stated in terms of the first derivative of $f$ and are discussed in detail in Section 5. Any meromorphic strictly monotone function $f$, such as in \cite{Bellissard}, satisfies these conditions, and one can state the following theorem in a self-contained form as a consequence of our results.
\begin{thm}
\label{th_main_meromorphic}
Let $f\colon (-1/2,1/2)\to \mathbb R$ be a strictly increasing function satisfying $f(-1/2+0)=-\infty$, $f(1/2-0)=+\infty$ and admitting a meromorphic $1$-periodic extension to a neighborhood of the real line in $\mathbb C$. Suppose that the frequency vector $\omega$ satisfies the Diophantine condition $\eqref{eq_diophantine_definition}$
$$
\|\bn \cdot\omega\|=\dist(\bn\cdot\omega,\Z)\ge C_{\mathrm{dio}}|\bn|^{-\tau},\quad \forall \bn\in \Z^d\setminus\{\bze\}.
$$
Let $H(x_0)$ be the operator \eqref{eq_h_def} with $\varphi_{\bn}\neq 0$ for finitely many points $\bn\in\Z^d$. There exists $\ep_0=\ep_0(d,f,\varphi,C_{\dio},\tau)>0$ such that, for any $\ep\in[0,\ep_0)$, any $x_0\in \R\setminus(\Z+1/2+\Z^d\cdot\omega)$, and any $\bn_0\in \Z^d$, the Rayleigh--Schr\"odinger perturbation series for the eigenvalues and eigenvectors of $H(x_0)$, started at $\bn_0$, converge and provide a complete orthogonal system of eigenvectors for $H(x_0)$, parametrized by $\bn_0\in\Z^d$.
\end{thm}
Note that the strict monotonicity of $f$ on $(-1/2,1/2)$ and rational independence of the components of $\omega$ (which follows from the Diophantine condition) imply that the potential in $H(x_0)$ satisfies the algebraic non-resonance condition at all lattice points for all $x_0\in \R\setminus(\Z+1/2+\Z^d\cdot\omega)$. In other words, $f(x_0+\bm\cdot\omega)\neq f(x_0+\bn\cdot\omega)$ for $\bm\neq \bn$.

For the remaining results summarized below, it is more convenient to postpone the exact statements until further sections.
\begin{enumerate}
	\item The main result, from which Theorem \ref{th_main_meromorphic} follows, is Theorem \ref{main} whose proof is completed in Section 5. The conditions on $f$, in both local and global forms, are described in Section 5.1. These conditions, referred to as $C_{\reg}$-regularity, can be described entirely in terms of the first derivatives of $f$.
	\item Several generalizations are discussed in Section 6. In particular, the hopping terms $\varphi_{\bn}$ can be allowed to depend on $x$ in a quasiperiodic way, and infinite range hopping is allowed, with some natural assumptions that higher range hopping terms have larger powers of $\ep$ in front of them. A somewhat technical case is also considered where $f$, while still strictly monotone, is allowed to have intervals on which the lower bound on the derivative is, essentially, $\ep$-dependent. This case serves as a preparation for the next section. A version of the construction from Section 6 is also used in our subsequent paper \cite{KPSK}.
	\item In the last section, we consider a model example of $f$ which is not necessarily strictly monotone on $(-1/2,1/2)$ and is allowed to have a flat segment. As a consequence, the algebraic non-resonance condition is violated for this operator. We describe a scheme that allows to transform such operators to those satisfying the assumptions of a general theorem from Section 6. This allows us to obtain Anderson localization for these operators.
\end{enumerate}

\section{Graphical and symbolic representations of the perturbation series}
In this section, we will discuss in detail the construction of the perturbation series. For simplicity of notation, we will always assume that the series is started at $\bn_0=0$. The general case can be easily considered by translation (or, in the case of quasiperiodic operators, by changing the phase $x_0$).
\subsection{Representation in terms of paths on a graph}
There are multiple ways of representing terms in the series \eqref{def_rs_series}. For our purposes, the most convenient way would be to associate them with paths on a certain weighted graph $\Gamma$. {\it The vertices of $\Gamma$ are defined as follows}. First, take a copy of $\mathbb Z^d$ and call it the sheet of height $0$ (``base sheet''). For each vertex $\bk_0$ of $\mathbb Z^d\setminus\{\bze\}$, we add a new separate copy of $\mathbb Z^d\setminus\{\bze\}$, with our vertex $\bk_0$ placed as the origin (so that $\bk_0$ together with a copy of $\mathbb Z^d\setminus\{\bze\}$ form a copy of  $\mathbb Z^d$). Each of these copies of $\mathbb Z^d\setminus\{\bze\}$ will be called sheets of height $1$. Now, add a separate copy of $\mathbb Z^d\setminus\{\bze\}$ for each vertex $\bk_1$ on each sheet of height $1$, and call these new copies sheets of height $2$. The result of this process, repeated indefinitely, will be the set of vertices of the graph $\Gamma$. One can also enumerate the vertices of the graph in a more direct way: each non-zero vertex of $\Gamma$ on a sheet of height $s$ can be associated with a sequence $\bk_0,\bk_1,\ldots,\bk_s$, where $\bk_0,\bk_1,\ldots,\bk_s\in \mathbb Z^d\setminus\{\bze\}$. If $v$ is a vertex of $\Gamma$ that is the point $\bk$ from one of the copies of $\Z^d$ or $\Z^d\setminus\{\bze\}$, we will call $\bk$ the {\it coordinate} of $v$. In other words, a vertex $v$ represented by a sequence $\bk_0,\bk_1,\ldots,\bk_s$ has coordinate $\bk_s$. We will also consider the origin (which only exists on the base sheet) to be the only vertex with coordinate $\bze$ and denote it by the same symbol $\bze$ as the origin in $\Z^d$. 

We will say that the vertex associated with a sequence $\bk_0,\bk_1,\ldots,\bk_{s-1},\bk_s$ is {\it directly above} the vertex associated to $\bk_0,\bk_1,\ldots,\bk_{s-1}$. We will also say that the vertex associated with $\bk_0,\bk_1,\ldots,\bk_{s},\ldots,\bk_t$ is above the vertex associated with $\bk_0,\bk_1,\ldots,\bk_{s}$. In the remaining text, we will avoid using the sequence notation for vertices, but we will often use the terms ``above'' or ``directly above''. We will also use the words ``the sheet of $\Gamma$ directly above the vertex $v$'' in the obvious interpretation. Sometimes it is also convenient to use words ``below'' or ``directly below'', whose meaning is opposite to ``above'' and ``directly above''. Note that, by construction, none of the vertices of $\Gamma$ are above or below $\bze$.

{\it The edges of the graph $\Gamma$ are defined as follows}. In all cases, $\bn$ and $\bn'$ are coordinates of two vertices $v,v'$. The graph $\Gamma$ is oriented; however, with each edge it will also contain the edge in the opposite direction (which may have a different weight --- see below). If $v$ and $v'$ are on the same sheet, there is an edge between $v$ and $v'$ if and only if $\Phi_{\bn'\bn}\neq 0$. That includes the origin on the base sheet (recall that other sheets do not have an origin on them). In addition, if $v'$ is located on the sheet directly above $v$, then there is an edge from $v$ to $v'$ if and only if $\Phi_{\bn'\bze}\neq 0$. In a way, $v$ plays the role of the origin in the sheet that is directly above $v$.

{\it To each edge of the graph $\Gamma$, we will associate a weight:}
\begin{itemize}
	\item If $v$ and $v'$ are on the same sheet, then the edge from $v$ to $v'$ will have weight $\frac{\Phi_{\bn'\bn}}{V_{\bze}-V_{\bn'}}$ if $\bn'\neq \bze$, and weight $\Phi_{\bn\bze}$ if $\bn'=\bze$; the latter can happen only on the base sheet.
\item If $v'$ is directly above $v$, then the ``ascending'' edge from $v$ to $v'$ has weight $\frac{\Phi_{\bn'\bze}}{V_{\bze}-V_{\bn'}}$, and the ``descending'' edge from $v'$ to $v$ has weight $-\frac{\Phi_{\bze\bn'}}{V_{\bze}-V_{\bn}}$.
\end{itemize}
Just like on any oriented graph, one can define a path on $\Gamma$ as a sequence of (not necessarily distinct) edges $(e_1,\ldots,e_p)$, where the edge $e_{j}$ ends at the starting vertex of $e_{j+1}$. We will say that this path starts from the starting vertex of $e_1$ and ends at the ending vertex of $e_p$. We will say that a path visits a vertex $v$ if one of the edges $e_1,\ldots,e_p$ starts or ends with $v$. A path may visit the same vertex multiple times.

We will consider two types of paths on $\Gamma$. An {\it eigenvalue path} is a path that starts and ends at the origin and never visits the origin in between. An {\it eigenvector path} starts at the origin, ends at some non-zero vertex with coordinate $\bn$ on the base sheet, and does not visit the origin in between. It will also be convenient to consider eigenvalue paths that do not start from the origin. By definition, a {\it non-base $($eigenvalue$)$ path} is a path on $\Gamma$ that starts and ends from the same vertex $v$ and only visits vertices above $v$ in between. An eigenvalue path would be a non-base path that starts from the origin. Since the graph $\Gamma$ above any point looks identical, one can translate any eigenvalue path into any vertex of $\Gamma$ and obtain a non-base path. We will not consider non-base eigenvector paths. 

Since we will mostly consider eigenvalue/eigenvector paths, {\it any claim about paths, by default, is applicable only to eigenvalue/eigenvector paths}. We will always specify if we are talking specifically about non-base paths, eigenvalue paths, or eigenvector paths.

In all cases, $|\CP|$ will denote the length of $\CP$ (in other words, the number of edges of $\CP$). By $\cont(\CP)$, we denote the product of weights of all edges of $\CP$. If a path travels along the same edge multiple times, each of them gives a separate contribution.

If $\CQ$ is an eigenvalue path and $\CQ'$ is a non-base path obtained from $\CQ$ by translating it into a vertex with coordinate $\bk$, then we have
\bee
\label{eq_nonbase_weight}
\cont(\CQ')=-(V_{\bze}-V_{\bk})^{-1}\cont(\CQ),
\ene
since the weight of the last descending edge in $\CQ'$ has an extra factor $-(V_{\bze}-V_{\bk})^{-1}$.
\begin{rem}
\label{rem_zero_weights} In the above construction, $\Gamma_j$ has an edge between two points with coordinates $\bm$, $\bn$ on the same sheet if and only if $\Phi_{\bn\bm}\neq 0$. In some constructions, it is convenient to assume that has edges between {\it any} two points on the same sheet, with weights defined in the same way as above. At this stage, it is just a formality since newly added edges will have zero weights and no contribution to any calculations. Later, when one considers $x$-dependent hopping terms, it will be convenient that the $x$-dependence is only present in the edge weights but not in the graph $\Gamma$.
\end{rem}
The following is a formulation of the Rayleigh--Schr\"odinger perturbation theory. See, for example, \cite{Arnold} for a similar result in a different notation and slightly different setting.
\bet
\label{th_lak} 
Define $V$ and $\Phi$ as in \eqref{eq_v_def}, \eqref{eq_phi_def}, and assume that the algebraic non-resonance condition $\eqref{eq_v_nonresonant}$ is satisfied at $\bn_0=\bze$. Define $\Gamma$ and eigenvalue/eigenvector paths on $\Gamma$ as above. Then the infinite system of equations
\bee
\label{eq_rs_series}
(V+\ep \Phi)(\psi_0+\ep\psi_1+...)=(\la_0+\ep\la_1+...)(\psi_0+\ep\psi_1+...),
\ene
treated formally by equating left and right hand side at each power of $\varepsilon$, with 
\bee
\label{eq_rs_conditions}
\lambda_s\in \R,\quad \psi_s\in \ell^2(\Z^d), \quad \lambda_0=V_{\bze},\quad (\psi_s)_{\bze}=0,\,\,s>0;\quad \psi_0=e_{\bze},
\ene
has a unique solution given by
\bee
\label{eq_lak}
\la_s=\sum_{\CP\colon |\CP|=s} \cont(\CP),
\ene
where the sum is considered over all eigenvalue paths $\CP$ on $\Gamma$ $($with $|\CP|=s)$, and
\bee
\label{eq_psik}
(\psi_s)_{\bk}=\sum_{\CP\colon|\CP|=s} \cont(\CP),\,\, \bk\neq 0,\,\, s>0,
\ene
where the sum is considered over all eigenvector paths $\CP$  $($again, with $|\CP|=s)$ between vertices with coordinates $\bze$ and $\bk$ on the base sheet of $\Gamma$.
\ent 
\begin{proof}
The first equation obtained from \eqref{eq_rs_series} by considering terms with $\varepsilon^0$, is already contained in \eqref{eq_rs_conditions}. The next equation (corresponding to $\varepsilon^1$) leads to
\bee
\label{eq_rs_first}
V\psi_1+\Phi e_{\bze}=V_{\bze}\psi_1+\lambda_1\psi_0.
\ene
We will solve this equation by projecting the left and right hand sides onto $\mathrm{span}\{e_{\bze}\}$ and $\mathrm{span}\{e_{\bze}\}^{\perp}$. It will be convenient to use the ``partial inverse'' to the non-invertible operator $V_{\bze}I-V$:
$$
(V_{\bze}-V)^{-1}:=(V_{\bze}I-V)^{-1}(1-\langle e_{\bze},\cdot\rangle e_{\bze}),
$$
that is, we always assume that  $(V_{\bze}-V)^{-1}$ is extended by zero into $\ker (V_{\bze}I-V)$. The operator $(V_{\bze}-V)^{-1}$, as well as $V$ itself, is a possibly unbounded multiplication operator. However, we will only apply $(V_{\bze}-V)^{-1}$ to vectors with finite support, which are always in its domain. With these conventions, \eqref{eq_rs_first} reduces to
$$
\lambda_1=0,\quad  \psi_1=(V_{\bze}-V)^{-1}\Phi e_{\bze}.
$$
The comparison of terms of \eqref{eq_rs_series} at $\varepsilon^2$ yields
\bee
\label{eq_rs_eps2}
V\psi_2 +\Phi\psi_1=V_{\bze}\psi_2+\lambda_1\psi_1+\lambda_2 e_{\bze}.
\ene
Projecting \eqref{eq_rs_eps2} onto $\mathrm{span}\{e_{\bze}\}$ and $\mathrm{span}\{e_{\bze}\}^{\perp}$ together with the orthogonality condition $\psi_j\perp e_{\bze}=0$ for $j>0$ lead to
$$
\lambda_2=\langle \Phi \psi_1,e_{\bze}\rangle,
$$
$$
\psi_2=(V_{\bze}-V)^{-1}\Phi\psi_1-\lambda_1 (V_{\bze}-V)^{-1} \psi_1=(V_{\bze}-V)^{-1}\Phi\psi_1
$$
In general, the equation at $\varepsilon^s$
$$
V\psi_s+\Phi \psi_{s-1}=V_{\bze}\psi_s+\lambda_1 \psi_{s-1}+\lambda_2\psi_{s-2}+\ldots+\lambda_{s-1}\psi_1+\lambda_{s}e_{\bze},
$$
after projecting it onto the same subspaces, becomes a system of two equations
\bee
\label{eq_rs_psik}
\psi_{s}=(V_{\bze}-V)^{-1}(\Phi\psi_{s-1}-\lambda_2\psi_{s-2}-\ldots-\lambda_{s-1}\psi_1)
\ene
and
\begin{multline}
\label{eq_rs_lak}	
\lambda_{s}=\langle \Phi \psi_{s-1},e_{\bze}\rangle =\langle \Phi (V_{\bze}-V)^{-1}\Phi\psi_{s-2},e_{\bze}\rangle - \lambda_2 \langle \Phi (V_{\bze}-V)^{-1}\psi_{s-3},e_{\bze}\rangle\\-\lambda_3 \langle \Phi (V_{\bze}-V)^{-1}\psi_{s-4},e_{\bze}\rangle-\ldots
-\lambda_{s-2}\langle \Phi (V_{\bze}-V)^{-1}\psi_1,e_{\bze}\rangle.
\end{multline}
One can easily check that, say, for $s=1$ and $s=2$ this gives the same expressions as \eqref{eq_lak}, \eqref{eq_psik}. It remains to show that \eqref{eq_lak}, \eqref{eq_psik} satisfy recurrent relations similar to \eqref{eq_rs_lak}, \eqref{eq_rs_psik}. 

Consider the first equality in \eqref{eq_rs_lak}. In terms of paths, it corresponds to the following fact: each eigenvalue path ends at the origin, and the origin is only connected to the base sheet. Therefore, each eigenvalue path of length $s$ can be obtained from an eigenvector path of length $s-1$ that ends at a vertex with coordinate $\bk$ on the base sheet by adding an extra edge from $\bk$ to $\bze$. By the induction assumption \eqref{eq_psik} for $s-1$, the total contribution of such eigenvector paths is $(\psi_{s-1})_{\bk}$. The extra edge from $\bk$ to $\bze$ will produce a factor $\Phi_{\bze\bk}$, which results in
$$
\lambda_s=\sum_{\bk\in \Z^d\setminus\{\bze\}}\Phi_{\bze\bk}(\psi_{s-1})_{\bk},
$$
thus arriving to the first equality in \eqref{eq_rs_lak}.

To establish \eqref{eq_rs_psik}, suppose that $\CP$ is an eigenvector path that ends at $\bk$ on the base sheet. Then there are two possibilities: either it arrives to $\bk$ from some other vertex $\bm$ on the base sheet, or it descends into $\bk$ from the sheet directly above $\bk$.

In the first case, $\CP$ is obtained from some other eigenvector path $\CP'$ from $\bze$ to $\bm$ of length $s-1$, by adding an edge from $\bm$ to $\bk$ on that sheet. By the induction assumption \eqref{eq_psik} with $s-1$, the total contribution of such paths from $\bze$ to $\bm$ is $(\psi_{s-1})_{\bm}$. The weight of the extra edge from $\bm$ to $\bk$ is $\Phi_{\bk\bm}(V_{\bze}-V_{\bk})^{-1}$. After summing over all possible choices of $\bm$, we arrive to the total contribution of the paths from the first case being
$$
\sum_{\bm\in \Z^d\setminus\{\bze\}}(V_{\bze}-V_{\bk})^{-1}
\Phi_{\bk\bm}(\psi_{s-1})_{\bm},
$$
which is equal to the $\bk$th component of the first term in \eqref{eq_rs_psik}.

In the second case note that, in order to be able to descend into the vertex $\bk$ on the base sheet, the path $\CP$ should have, sometime earlier, ascended from the same vertex to the sheet above it. Consider the last time when $\CP$ ascended from $\bk$ on the base sheet and break $\CP$ into two parts: the part $\CP'$ right before that event, and the part $\CQ'$ after it.
The part $\CQ'$ is a non-base eigenvalue path that starts and ends at $\bk$ and spends the rest of time in the sheets above $\bk$. Its contribution can be calculated using \eqref{eq_nonbase_weight}, thus obtaining
$$
\cont(\CP)=\cont(\CQ')\cont(\CP')=-(V_{\bze}-V_{\bk})^{-1}\cont(\CQ)\cont(\CP'),
$$
where $\CQ$ is the eigenvalue path obtained by translating $\CQ'$ into the origin. By considering all possible choices for $\CQ$, we arrive to the total contribution of paths in this case being
$$
-\sum_{\CQ,\CP'}(V_{\bze}-V_{\bk})^{-1}\cont(\CQ)\cont(\CP'),
$$
where the summation is considered over all eigenvalue paths $\CQ$ with $|\CQ|=j$ and all eigenvector paths $\CP'$ from $\bze$ to $\bk$ of length $s-j$, with $j=2,3,\ldots,s-1$. For a fixed $j$, one can factorize the summation and apply the induction assumptions \eqref{eq_lak}, \eqref{eq_psik} at stages $j$ and $s-j$, respectively, and obtain that, for each $j$, the contribution of paths $\CP=\CQ'\CP'$ with $|\CQ'|=j$ is equal to
$$
-(V_{\bze}-V_{\bk})^{-1}\left(\sum_{\CQ:|\CQ|=j}\cont(\CQ)\right)\left(\sum_{\CP':|\CP'|=s-j}\cont(\CP)\right)=-(V_{\bze}-V_{\bk})^{-1}\lambda_j(\psi_{s-j})_{\bk},
$$
which can now be identified with a term of \eqref{eq_rs_psik}.
\end{proof}
\subsection{Symbolic and graphical representation of paths on $\Gamma$} 
We will use the following notation in order to describe paths on $\Gamma$. Let $\CP$ be an eigenvalue path which only travels along the base sheet of $\Gamma$ (later, such paths will be called {\it loops}). Suppose that it starts at the origin, then visits the vertices of the base sheet with coordinates $\bn_1,\bn_2,\ldots,\bn_k\neq \bze$ (in this order, and each new visit of the same vertex is accounted for separately), and then returns to the origin. We will use the following notation for this path:
$$
\CP=(\bn_1\bn_2\ldots\bn_k)
$$
We will also denote eigenvector paths in the same way, but will include the last point in the notation:
$$
\CP=(\bn_1\bn_2\ldots\bn_k\bm),
$$
where $\bm\notin \{\bn_k,\bze\}$ is the ending point of the path $\CP$. Whether the string denotes an eigenvalue or an eigenvector path, should be specified in the context, otherwise the notation is ambiguous: for example, the string $(1234321)$ denotes both an eigenvalue path and an eigenvector path that starts at $0$ and ends at $1$ (in this example, we assumed that all coordinates belong to $\mathbb Z$, that is, $d=1$, and are not using boldface notation for them).

Whenever a path ascends to a sheet above some vertex, we will write an opening parenthesis, and then continue with coordinates of vertices that $\CP$ visits on the upper sheet. For example, the eigenvalue path
\bee
\label{eq_string_example}
\CP=(12345(12321)5434321)
\ene
travels from the origin to the vertex with coordinate 5 on the base sheet (here we again assume that all coordinates belong to $\mathbb Z$, that is, $d=1$), then uses an ascending edge to ascend the sheet of height 1 directly above the vertex 5, then travels to 3 and back to 1 on that sheet, and then descends to the base sheet and proceeds along the remaining segment $434321$. Whenever the path uses a descending edge to descend to a lower sheet, we use the closing parenthesis. One can have multiple levels of parentheses, depending on how high the path climbs. Given the string representing $\CP$, one can calculate $\cont(\CP)$ using the following rules, which will associate an edge to each pair of consecutive coordinates in that string. Suppose $v$ and $v'$ are vertices with coordinates $\bn,\bn'$ respectively.
\begin{itemize}
	\item $\bn \bn'$ represents the edge from $v_1$ to $v_2$ assuming that they are on the same sheet.
	\item $\bn(\bn'$ represents the ascending edge from $v_1$ to $v_2$ assuming that $v_2$ is on a sheet directly above $v_1$.
	\item $\bn)\bn'$ represents the descending edge from $v_1$ to $v_2$ assuming that $v_1$ is on a sheet directly above $v_2$.
	\item $(\bn$ in the beginning of the string represents the edge from $\bze$ to $\bn$.
	\item if $\CP$ is an eigenvalue path, then $\bn)$ in the end of the string represents the edge from $v$ to $\bze$. In case of an eigenvector path, it does not represent anything (one can consider it as an edge of weight $1$ and length zero, meaning that it does not contribute to $|\CP|$ and $\cont(\CP)$).
\end{itemize}
To calculate $\cont(\CP)$, one can multiply the weights of the edges for each pair of consecutive lattice points in the string, using the rules from the above paragraph and weights from Section 3.1.

In order to better understand the constructions, we will also use graphical representations of the paths on $\Gamma$. A simple path $(1234321)$ will be represented graphically as the following diagram:
$$
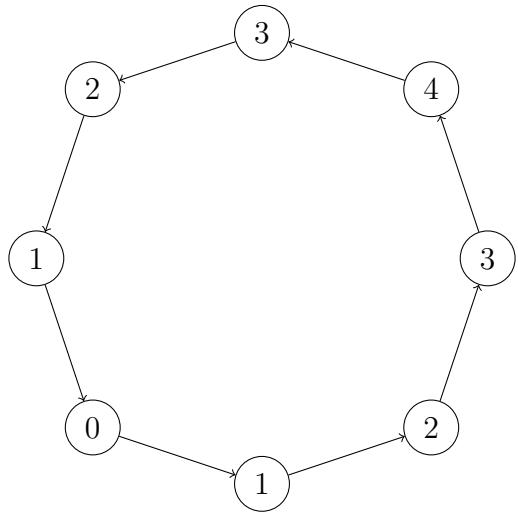
\begin{tikzpicture}
	\begin{pgfonlayer}{nodelayer}
		\node [style=Vertex] (0) at (0, -3) {$1$};
		\node [style=Vertex] (1) at (-3, 0) {$1$};
		\node [style=Vertex] (2) at (0, 3) {$3$};
		\node [style=Vertex] (3) at (3, 0) {$3$};
		\node [style=Vertex] (4) at (2.25, -2.25) {$2$};
		\node [style=Vertex] (5) at (-2.25, -2.25) {$0$};
		\node [style=Vertex] (6) at (2.25, 2.25) {$4$};
		\node [style=Vertex] (7) at (-2.25, 2.25) {$2$};
	\end{pgfonlayer}
	\begin{pgfonlayer}{edgelayer}
		\draw [style=Arrow edge] (5) to (0);
		\draw [style=Arrow edge] (0) to (4);
		\draw [style=Arrow edge] (4) to (3);
		\draw [style=Arrow edge] (3) to (6);
		\draw [style=Arrow edge] (6) to (2);
		\draw [style=Arrow edge] (2) to (7);
		\draw [style=Arrow edge] (7) to (1);
		\draw [style=Arrow edge] (1) to (5);
	\end{pgfonlayer}
\end{tikzpicture}
$$
A path visiting higher sheets of $\Gamma$, such as $(12345(12321)5(121)54321)$, corresponds to the  diagram
$$
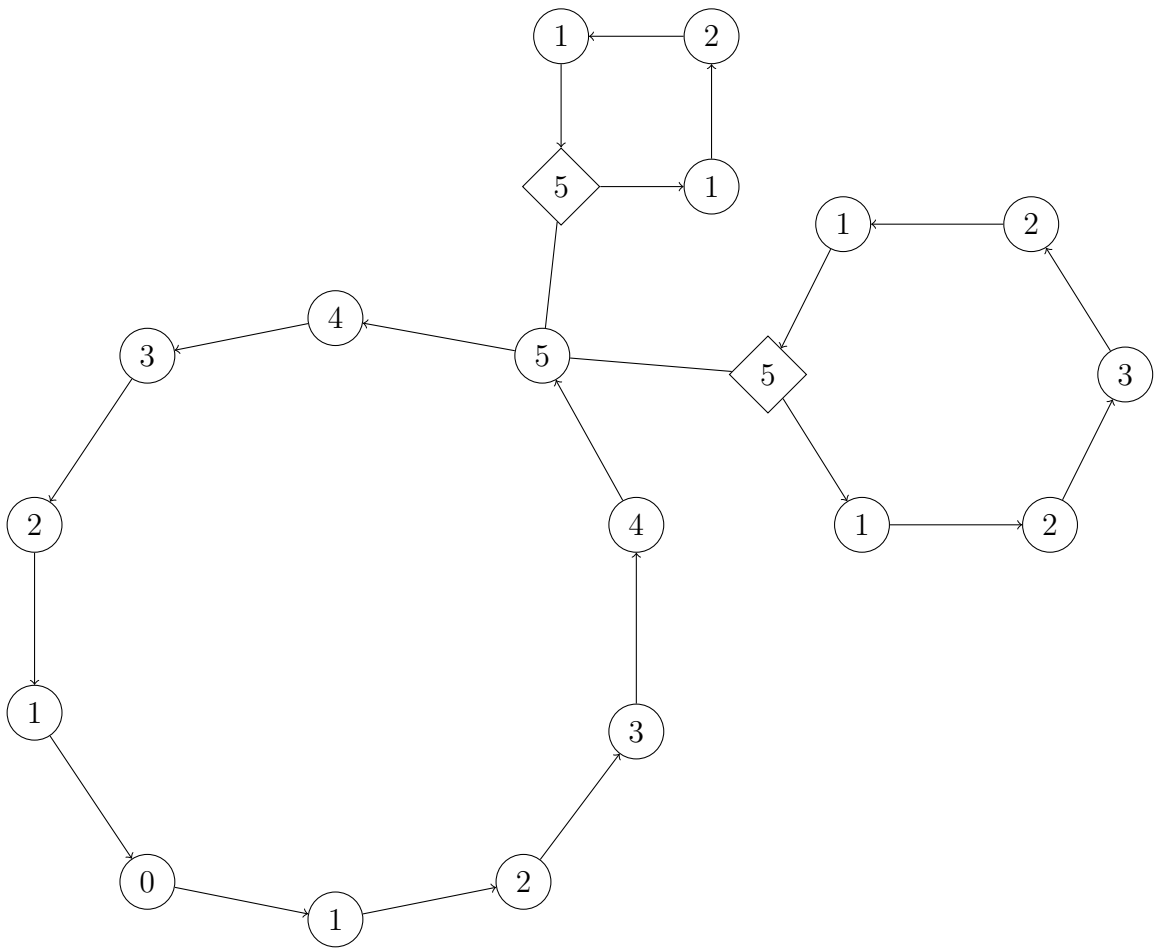
\begin{tikzpicture}
	\begin{pgfonlayer}{nodelayer}
		\node [style=Vertex] (0) at (0, -4) {$1$};
		\node [style=Vertex] (1) at (0, 4) {$4$};
		\node [style=Vertex] (2) at (4, 1.25) {$4$};
		\node [style=Vertex] (3) at (4, -1.5) {$3$};
		\node [style=Vertex] (4) at (2.5, -3.5) {$2$};
		\node [style=Vertex] (5) at (2.75, 3.5) {$5$};
		\node [style=Vertex] (6) at (-4, 1.25) {$2$};
		\node [style=Vertex] (7) at (-4, -1.25) {$1$};
		\node [style=Vertex] (8) at (-2.5, -3.5) {$0$};
		\node [style=Vertex] (9) at (-2.5, 3.5) {$3$};
		\node [style=base vertex] (10) at (3, 5.75) {$5$};
		\node [style=Vertex] (11) at (5, 5.75) {$1$};
		\node [style=Vertex] (12) at (5, 7.75) {$2$};
		\node [style=Vertex] (13) at (3, 7.75) {$1$};
		\node [style=base vertex] (14) at (5.75, 3.25) {$5$};
		\node [style=Vertex] (15) at (7, 1.25) {$1$};
		\node [style=Vertex] (16) at (9.5, 1.25) {$2$};
		\node [style=Vertex] (17) at (10.5, 3.25) {$3$};
		\node [style=Vertex] (18) at (9.25, 5.25) {$2$};
		\node [style=Vertex] (19) at (6.75, 5.25) {$1$};
	\end{pgfonlayer}
	\begin{pgfonlayer}{edgelayer}
		\draw [style=Arrow edge] (8) to (0);
		\draw [style=Arrow edge] (0) to (4);
		\draw [style=Arrow edge] (4) to (3);
		\draw [style=Arrow edge] (3) to (2);
		\draw [style=Arrow edge] (2) to (5);
		\draw [style=Arrow edge] (5) to (1);
		\draw [style=Arrow edge] (1) to (9);
		\draw [style=Arrow edge] (9) to (6);
		\draw [style=Arrow edge] (6) to (7);
		\draw [style=Arrow edge] (7) to (8);
		\draw [style=Arrow edge] (10) to (11);
		\draw [style=Arrow edge] (11) to (12);
		\draw [style=Arrow edge] (12) to (13);
		\draw [style=Arrow edge] (13) to (10);
		\draw (5) to (10);
		\draw (5) to (14);
		\draw [style=Arrow edge] (14) to (15);
		\draw [style=Arrow edge] (15) to (16);
		\draw [style=Arrow edge] (16) to (17);
		\draw [style=Arrow edge] (17) to (18);
		\draw [style=Arrow edge] (18) to (19);
		\draw [style=Arrow edge] (19) to (14);
	\end{pgfonlayer}
\end{tikzpicture}
$$
In each diagram, an oriented edge corresponds to an edge of the original path. For the convenience of drawing, each time a path enters a new sheet of $\Gamma$ directly above a vertex with coordinate $\bn$, we create a copy of that vertex on the picture and draw the corresponding part of the path starting from that copy. Such "extra" copy is represented by a diamond-shaped node. The oriented edges connected to that copy are the corresponding ascending and descending edges of $\Gamma$. The non-oriented edge connecting $\bn$ and the diamond $\bn$ does not correspond to any edge of $\Gamma$ and only represents the fact that these nodes correspond to the same vertex of the graph.

In order to recover a path from its diagram, one can start from the node representing the origin (there is only one such node) and then follow the arrows. In the event a node with coordinate $\bn$ is connected with a diamond-shaped node, upon arrival into $\bn$ one must first go to the diamond-shaped node and follow the arrows from that node; they will eventually return back to the diamond-shaped node. Afterwards, one can continue following the arrow coming out of circle-shaped $\bn$. In case $\bn$ is connected to multiple diamond-shaped nodes, as in the last example, one needs to treat them all in counter-clockwise order before proceeding out of $\bn$. In particular, this implies that the order in which different attached loops appear on the diagram is important, since the opposite attachment would correspond to a different string $(12345(121)5(12321)54321)$.

The contribution $\cont(\CP)$ can be recovered from the picture by using the following rules, representing the fact that each oriented edge of $\Gamma$ corresponds to an oriented edge on a picture and gives a multiplicative contribution. In the below notation, we assume $\bm,\bn\neq \bze$. We also include non-oriented edges in this table for completeness, emphasizing the fact that they do not correspond to edges of $\Gamma$ and give no contribution.
$$
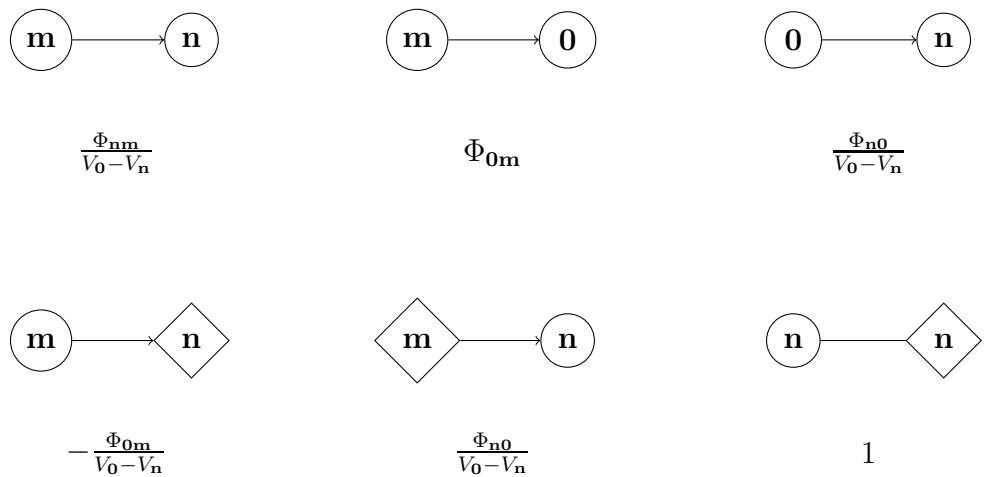
\begin{tikzpicture}
	\begin{pgfonlayer}{nodelayer}
		\node [style=Vertex] (0) at (-7, 1) {$\bm$};
		\node [style=Vertex] (1) at (-5, 1) {};
		\node [style=Vertex] (2) at (-5, 1) {$\bn$};
		\node [style=none] (3) at (-6, -0.5) {};
		\node [style=none] (4) at (-6, -0.5) {$\frac{\Phi_{\bn\bm}}{V_{\bze}-V_{\bn}}$};
		\node [style=Vertex] (5) at (-2, 1) {$\bm$};
		\node [style=Vertex] (6) at (0, 1) {};
		\node [style=Vertex] (7) at (0, 1) {$\bze$};
		\node [style=none] (8) at (-1, -0.5) {};
		\node [style=none] (9) at (-1, -0.5) {$\Phi_{\bze\bm}$};
		\node [style=Vertex] (10) at (3, 1) {$\bze$};
		\node [style=Vertex] (11) at (5, 1) {$\bn$};
		\node [style=none] (12) at (4, -0.5) {};
		\node [style=none] (13) at (4, -0.5) {$\frac{\Phi_{\bn\bze}}{V_{\bze}-V_{\bn}}$};
		\node [style=Vertex] (14) at (-7, -3) {$\bm$};
		\node [style=base vertex] (15) at (-5, -3) {$\bn$};
		\node [style=none] (16) at (-6, -4.5) {$-\frac{\Phi_{\bze\bm}}{V_{\bze}-V_{\bn}}$};
		\node [style=Vertex] (17) at (-2, -3) {};
		\node [style=Vertex] (18) at (0, -3) {};
		\node [style=base vertex] (19) at (-2, -3) {$\bm$};
		\node [style=Vertex] (20) at (0, -3) {$\bn$};
		\node [style=none] (21) at (-1, -4.5) {};
		\node [style=none] (22) at (-1, -4.5) {$\frac{\Phi_{\bn\bze}}{V_{\bze}-V_{\bn}}$};
		\node [style=Vertex] (23) at (3, -3) {$\bn$};
		\node [style=base vertex] (24) at (5, -3) {$\bn$};
		\node [style=none] (25) at (4, -4.5) {};
		\node [style=none] (26) at (4, -4.5) {$1$};
	\end{pgfonlayer}
	\begin{pgfonlayer}{edgelayer}
		\draw [style=Arrow edge] (0) to (2);
		\draw [style=Arrow edge] (5) to (7);
		\draw [style=Arrow edge] (10) to (11);
		\draw [style=Arrow edge] (14) to (15);
		\draw [style=Arrow edge] (19) to (20);
		\draw (23) to (24);
	\end{pgfonlayer}
\end{tikzpicture}
$$

\subsection{Attachment of paths and loops} Recall that any eigenvalue path can be moved into any point of $\Gamma$, where it becomes a non-base path. Suppose that $\CP$ is a path and $v$ is a vertex on $\CP$. Let $\CQ$ be an eigenvalue path. Then, one can construct a new path in the following way: take the part of $\CP$ until it reaches $v$, then insert a copy of $\CQ$ moved into $v$ (which becomes a non-base path that starts and ends at $v$) and then, once the copy of $\CQ$ returns to $v$, continue following the remaining part of $\CP$. We will say that the new path is obtained by attaching $\CQ$ to $\CP$ at the vertex $v$. In the symbolic notation, suppose that $\bn$ is the coordinate of $v$. To attach $\CQ$, we replace $\bn$ by $\bn\CQ\bn$ at the corresponding position of $\bn$ in $\CP$. Note that $\CP$ can visit $v$ multiple times. In that case, we can attach $\CQ$ at different positions of $v$ on $\CP$ and, in general, obtain different paths.

By an {\it eigenvalue/eigenvector loop}, we will denote an eigenvalue/eigenvector path that does not leave the base sheet of $\Gamma$. It is easy to see that any path can be obtained from a loop of the same type by finitely many attachments of eigenvalue loops. For example, to obtain the eigenvector path
$$
(12345(123(1234321)321)543),
$$
graphically represented as
$$
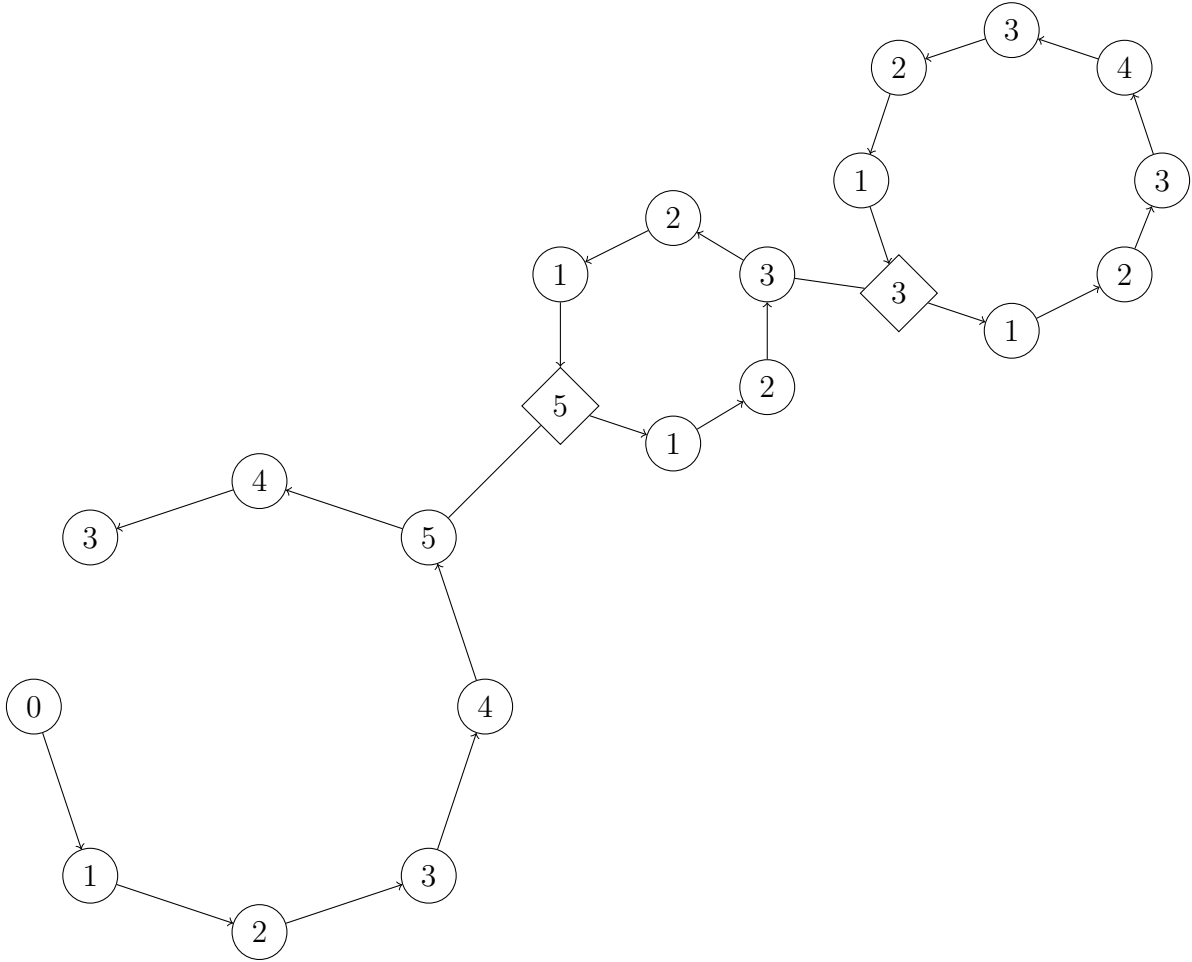
\begin{tikzpicture}
	\begin{pgfonlayer}{nodelayer}
		\node [style=Vertex] (0) at (-3, 0) {$0$};
		\node [style=Vertex] (1) at (3, 0) {$4$};
		\node [style=Vertex] (2) at (0, 3) {$4$};
		\node [style=Vertex] (3) at (0, -3) {$2$};
		\node [style=Vertex] (4) at (-2.25, 2.25) {$3$};
		\node [style=Vertex] (5) at (2.25, 2.25) {$5$};
		\node [style=Vertex] (6) at (2.25, -2.25) {$3$};
		\node [style=Vertex] (7) at (-2.25, -2.25) {$1$};
		\node [style=base vertex] (8) at (4, 4) {$5$};
		\node [style=Vertex] (9) at (6.75, 5.75) {$3$};
		\node [style=Vertex] (10) at (5.5, 6.5) {$2$};
		\node [style=Vertex] (12) at (4, 5.75) {$1$};
		\node [style=Vertex] (13) at (5.5, 3.5) {$1$};
		\node [style=Vertex] (14) at (6.75, 4.25) {$2$};
		\node [style=base vertex] (15) at (8.5, 5.5) {$3$};
		\node [style=Vertex] (16) at (8.5, 8.5) {$2$};
		\node [style=Vertex] (17) at (11.5, 8.5) {$4$};
		\node [style=Vertex] (18) at (11.5, 5.75) {$2$};
		\node [style=Vertex] (19) at (8, 7) {$1$};
		\node [style=Vertex] (20) at (10, 9) {$3$};
		\node [style=Vertex] (21) at (12, 7) {$3$};
		\node [style=Vertex] (22) at (10, 5) {$1$};
	\end{pgfonlayer}
	\begin{pgfonlayer}{edgelayer}
		\draw (5) to (8);
		\draw [style=Arrow edge] (0) to (7);
		\draw [style=Arrow edge] (7) to (3);
		\draw [style=Arrow edge] (3) to (6);
		\draw [style=Arrow edge] (6) to (1);
		\draw [style=Arrow edge] (1) to (5);
		\draw [style=Arrow edge] (5) to (2);
		\draw [style=Arrow edge] (2) to (4);
		\draw [style=Arrow edge] (8) to (13);
		\draw [style=Arrow edge] (13) to (14);
		\draw [style=Arrow edge] (14) to (9);
		\draw [style=Arrow edge] (9) to (10);
		\draw [style=Arrow edge] (10) to (12);
		\draw [style=Arrow edge] (12) to (8);
		\draw (9) to (15);
		\draw [style=Arrow edge] (15) to (22);
		\draw [style=Arrow edge] (22) to (18);
		\draw [style=Arrow edge] (18) to (21);
		\draw [style=Arrow edge] (21) to (17);
		\draw [style=Arrow edge] (17) to (20);
		\draw [style=Arrow edge] (20) to (16);
		\draw [style=Arrow edge] (16) to (19);
		\draw [style=Arrow edge] (19) to (15);
	\end{pgfonlayer}
\end{tikzpicture}
$$
we start from the base (eigenvector) loop $(1234543)$, attach $(12321)$ at the vertex with coordinate $5$, thus obtaining $(12345(12321)543)$, and then attach another eigenvalue loop $(1234321)$ to the vertex with coordinate $3$ in the middle. The base loop and all attached loops that were used in building $\CP$ will be called {\it loops on $\CP$}. Note that the difference between eigenvalue and eigenvector paths is only in the type of the base loop. In both cases, all attached loops must be of the eigenvalue type.

From \eqref{eq_nonbase_weight}, if $\CP'$ is obtained from $\CP$ by attaching an eigenvalue loop $\CL$ at $\bn$, then
$$
\cont(\CP')=\cont(\CP)\cont(\CL')=-(V_{\bze}-V_{\bn})^{-1}\cont(\CP)\cont(\CL),
$$
where $\CL'$ is the non-base loop obtained by moving $\CL$ into $\bn$.

\section{Small denominator expansion}
\subsection{Small denominators and resonances} Fix $x_0\in \R\setminus(\Z+1/2+\Z^d\cdot\omega)$ and consider the operator \eqref{eq_h_def}:
$$
(H(x_0)\psi)_{\bn}=\varepsilon\sum_{\bm\in \Z^d}\varphi_{\bn-\bm}\psi_{\bm}+f(x_0+\bn\cdot\omega)\psi_{\bn}.
$$
Section 2.1 provides a formal power series for eigenvalues and eigenvectors of the operator \eqref{eq_h_def}. Our goal is to establish convergence of this series, for some class of functions $f$. For general non-monotone $f$ and small $\varepsilon$, there are two main obstructions to this convergence: small denominators and resonances. A small denominator appears whenever $\|\bn\cdot\omega\|=\dist(\bn\cdot\omega,\Z)\ll \varepsilon$. If we assume some regularity of $f$, this means that $f(x_0)-f(x_0+\bn\cdot\omega)$ is small. A resonance is a situation when $f(x_0)-f(x_0+\bn\cdot\omega)$ is small for some other reason: for example, if $f$ has multiple intervals of monotonicity, it is possible for the values of $f$ at $x_0$ and $x_0+\bn\cdot\omega$ to be close without $\|\bn\cdot\omega\|$ being small. Unlike small denominators, whether or not translation by $\bn$ creates a resonance depends strongly on $x_0$. {\it Except for the last section, we will assume that there are no resonances} (for example, by considering $f$ that are monotone or locally monotone). Throughout the text, lattice points $\bn$ such that $\|\bn \cdot\omega\|=\dist(\bn\cdot\omega,\Z)$ is small, will also be called small denominators.

\begin{rem}
\label{rem_summary1}
Note that, if $\omega$ satisfies some Diophantine properties, then it takes a relatively large number of steps along the lattice to reach a small denominator. If $\varepsilon$ is small, we can hope to compensate the contribution to $\cont(\CP)$ from the small denominator by the number of $\varepsilon$ factors that we gain at each step. Moreover, if we are traveling between two different small denominators, similar Diophantine arguments guarantee that the path $\CP$ has to make sufficiently many steps in between. Later in the text, these paths will be called safe. The contributions from safe paths form an absolutely convergent series for which no cancellations are required. However, there are also unsafe situations: suppose that a path visits the same small denominator many times, for example, by going back and forth between the small denominator $n$ and the point $n+1$ (similar to earlier examples, we are considering a quasiperiodic operator on $\ell^2(\Z)$ and not using boldface font to denote lattice points). Then, each of these small trips gains $\varepsilon^2$ in the numerator and a factor $f(x_0)-f(x_0+n\omega)$ in the denominator. If the latter factor is very small, the sum of these contributions diverges. Fortunately, in the quasiperiodic setting, the contribution of this path almost cancels with several other paths of the same length. Since these situations can happen on multiple levels simultaneously, {\it the goal of this section is to identify which paths are to be grouped together in order to take the biggest advantage of these cancellations}.
\end{rem}

%In this section, we describe a refined version of the expansion in Theorem \ref{th_lak}. We will classify the denominators $V_{\bn}$ by their ``smallness'', and will call a loop safe if it does not visit the same small denominator too often. To deal with non-safe loops, we will split them into smaller loops of new type, with each of the smaller loops being safe. This does not automatically solve the problem with small denominators, however, the only problem now is the extra factors that appear at the connection points of new loops. This ``localization'' of small denominators allows to develop the combinatorics of the cancellation procedure later in this section.

\subsection{Canonical marking and the equivalence relation} To begin with, we introduce the scale of ``the smallness'' of denominators. We will use the distance function on $\Z^d$ defined by the hopping matrix $\Phi$:
for $\bk,\bk'\in \Z^d$, let $\dist_{\varphi}(\bk,\bk')$ be the smallest number $l$ such that there exists a sequence $\bk=\bk_0,\bk_1,\ldots,\bk_l=\bk'$ with $\Phi_{\bk_{j+1}\bk_j}\neq 0$ for $0\le j\le l-1$; in other words, this is the shortest distance function on the base sheet of the graph $\Gamma$. We will say that two functions
\bee
\level\colon \Z^d\to \Z_+\cup\{+\infty\},\quad \safedist\colon \Z_+\to \Z_+,
\ene
form a {\it consistent denominator data} if the following is true:
\begin{itemize}
	\item[(c0)] $\level(\bze)=+\infty$, $\safedist(0)=0$. The function $\safedist$ is monotone non-decreasing in its argument.
	\item[(c1)] $\dist_{\varphi}(\bm,\bn)\ge \min\{\safedist(\level(\bm)),\safedist(\level(\bn))\}$, for $\bm\neq \bn$.
	\item[(c2)] Suppose, $0<\dist_{\varphi}(\bn,\bm)<\safedist(\level(\bm))$. Then $	\level(\bn-\bm)=\level(\bn)$.
\end{itemize}
Lattice points $\bm\neq \bze$ with $\level(\bm)>0$ will be called {\it small denominators}.
In order to simplify the notation, we will also extend the function $\safedist$ into $\Z^d$ by
$$
\safedist(\bn):=\safedist(\level(\bn));\quad \safedist(\bze):=+\infty.
$$
Later, we will give specific examples of consistent denominator data in the context of quasiperiodic operators. Throughout the rest of the section, most of the definitions will depend on the choice of a consistent denominator data, and we will assume that some such choice has been fixed.

%\begin{defn}
%Fix a consistent denominator data. An eigenvalue or eigenvector path on $\mathbb Z^d$ is called {\it safe}, if, between two consecutive visits of a point $(\bn_0\ldots,\bn_l)\in \Gamma$, it either does not visit the sheet containing that point, or makes at least $\safedist(\bn_l)$ steps on that sheet.
%\end{defn}
%\begin{rem}
%The first condition of the definition is reserved for the case when the path arrives to $(\bn_0\ldots,\bn_l)$, then immediately goes ascendings, spends some time in the above sheets, and then comes back to $(\bn_0\ldots,\bn_l)$. This would not be counted as two consecutive visits of that point.
%\end{rem}
\begin{defn}
\label{def_safe_loop}
Let $\CL$ be an eigenvalue/eigenvector loop. We say that $\CL$ is {\it safe}, if, between any two visits of a small denominator $\bm$, it makes at least $\safedist(\bm)$ steps.
\end{defn}
We will now describe an additional structure on the set of eigenvalue/eigenvector paths that will allow us to set up the cancellation procedure. Let $\CL$ be an eigenvalue/eigenvector loop. Suppose, $\bm$ is a small denominator visited by $\CL$ multiple times. Consider a segment of the string that defines $\CL$ between two consecutive visits of $\bm$ by $\CL$. For each such segment, we now choose either to mark it, in which case we put square parenthesis around that segment, or not to mark it.
For example, if $\CL=(12345432321)$ and $3$ is the only small denominator on $\CL$, then there are four possibilities of marking $\CL$:
$$
(12345432321),\quad(123[454]32321),\quad (1234543[2]321),\quad (123[454]3[2]321).
$$
The same entry of the string may belong to multiple marked segments. We will require the following: if two marked segments overlap, then one of them must be contained in another. As a consequence, the positions of the square parenthesis uniquely determines which segments are marked and which are not. Note that the marking is always applied to particular segments of $\CL$ rather than to corresponding vertices/edges of $\Gamma$. In particular, if a marked segment of $\CL$ contains a vertex $v$, it does not automatically imply that all consecutive visits of $v$ by $\CL$ will be marked.

Under the above assumptions, if $\CL$ is a loop with markings, one can consider a smaller loop obtained from $\CL$ by removing all marked segments from it, that is, replacing each marked segment of the form $\bm[\ldots]\bm$ by $\bm$ (always starting from the shortest segment). For example, in the above four cases the removal procedure would result in
$$
(12345432321),\quad(1232321),\quad (123454321),\quad (12321).
$$
Note that a copy of the small denominator also gets removed from the string each time, so that the loop remains well-defined. 

We now define {\it the canonical marking} of a loop $\CL$. The goal of it is the following: whenever $\CL$ has two visits of a small denominator $\bm$ in a short time, we mark the segment between these visits. By ``short'', we mean shorter than $\safedist(\bm)$. However, if $\CL$ visits another small denominator $\bm'$ of smaller level between visits of $\bm$, and the segment between visits of $\bm'$ is also short, then we do not want to include it in the calculation of time spent between visits of $\bm$. More formally, do the following procedure, starting from an eigenvalue/eigenvector loop $\CL$.
\begin{enumerate}
\item For each small denominator $\bm$ of level 1 and each segment between two consecutive visits of $\bm$ of length $<\safedist(1)$, mark that segment on $\CL$.
\item Denote by $\CL'$ the result of removing all marked segments on $\CL$.
\item For each small denominator $\bm'$ of level 2 on $\CL'$ and each segment between two consecutive visits of $\bm'$ of length $<\safedist(2)$ on $\CL'$, mark the corresponding segment on $\CL$.
\item Denote by $\CL''$ the result of removing all previously marked segments on $\CL$.
\item Repeat for levels 3,4,5,...
\end{enumerate}
In other words, we mark every ``short'' segment between two consecutive visits of each small denominator. The notion of length used in defining ``short'' should not count already marked segments.
\begin{rem}
Assumption (c1) implies that the above procedure will actually produce a loop with markings; in other words, that for any two overlapping marked segments one of them will be contained in the other.
\end{rem}

If $\CL$ is safe, then the canonical marking of $\CL$ coincides with $\CL$ and it does not have any marked segments. If one removes all marked segments from the canonical marking of $\CL$, the resulting (shorter) loop will be safe. We would like to draw the reader's attention to the following: suppose that $\CL$ is not safe, but the definition of ``safe'' only fails for one segment between two visits of a denominator, say, of level $1$. It is still possible that more than one segment is marked, because the removal of first segment may shorten segments between denominators of higher levels. Therefore, it is important that the algorithm proceeds from lower to higher levels and not the other way around (although it is possible to modify it appropriately).

We will now define the {\it canonical translation} of a loop. We first define it on the level of strings of symbols and will later check that it defines a path on $\Gamma$. Let $\CL$ be an eigenvalue/eigenvector loop, and suppose that it is canonically marked. 
\begin{enumerate}
\item For each vertex $v$ of $\CL$ with coordinate $\bn$, consider its position in the canonical marking of $\CL$. Suppose that $\bm[\ldots]\bm$ is the smallest marked segment that contains $v$. Then replace that entry $\bn$ by $\bn-\bm$. If $\bn$ is not on a marked segment, do nothing.
\item Replace all square parenthesis in the string by round parenthesis.
\end{enumerate}

From (c2), it follows that none of the denominators will change their level after translation. As a consequence, no point with non-zero coordinate will be translated into the origin, and therefore 
the above string will represent a path on $\Gamma$. We will denote this path by $T(\CL)$ and call it the {\it canonical translation} of $\CL$. Unless $\CL$ was safe, $T(\CL)$ will no longer be a loop, as it will involve visits to non-base sheets (encoded by the round parenthesis).

In other words, we are replacing each entry of the form $\bm[\ldots]\bm$ by $\bm(\ldots)\bm$, where everything between the parenthesis is translated by $-\bm$, unless it is inside some smaller marked segment. For example, suppose that $3$ and $6$ are small denominators, and the canonical marking of $\CL$ is $\CL=(123[456[5]654]321)$. We will use the same notation for $\CL$ and its canonical marking. Then
$$
T(\CL)=(123(123(-1)321)321).
$$
In this case, the entry $5$ in the middle is translated by $-6$, and the strings $456$ and $654$ are each translated by $-3$.

We have defined the canonical translation $T(\CL)$ of a loop $\CL$. Since any path on $\Gamma$ can be obtained from loops using the attachment procedure, we naturally define $T(\CP)$ by applying $T$ to each loop in $\CP$, and preserving the attachments at the same locations (based on their positions in the string of symbols).
\begin{rem}
Let $\CL$ be a loop and $T(\CL)$ be its canonical translation. As mentioned above, $T(\CL)$ is a path with multiple loops, but every loop of $T(\CL)$ is safe. Moreover, the base loop of $T(\CL)$ is the result of removing of all marked segments from the canonical marking of $\CL$.
\end{rem}
\begin{rem}
One can also define canonical marking and canonical translation directly for paths $\CP$. If $\CP$ visits a vertex $v$ on a sheet $S$ twice, we will say that the visits are consecutive if, between these visits, there are no other visits of $v$ and there are no visits of the sheet directly below $S$. In this case, in order to determine whether or not $\CP$ has made sufficiently many steps between these visits, we do not count previously marked segments and visits of sheets above $S$. Similarly, one needs to start from denominators of level $1$ and gradually add markings whenever the part of the path between visits is too short. In the translation procedure, for each vertex $v$, we determine the smallest marked segment on the same sheet which contains $v$ and translate that segment by the coordinate of its endpoints.
\end{rem}
\begin{defn}
\label{def_trans_eq}
Two eigenvalue/eigenvector paths $\CP$ and $\CQ$ are called {\it translationally equivalent} if $T(\CP)=T(\CQ)$.
Denote by $[\CP]$ the equivalence class of $\CP$, and let
$$
\cont([\CP]):=\sum_{\CP\in [\CP]}\cont(\CP).
$$
\end{defn}
\begin{rem}
\label{rem_summary2}
Grouping the paths into translational equivalence classes is one of the central ideas of the paper. In all applications, the cancellations will happen inside the equivalence classes. In fact, the proof of the convergence of the perturbation series 
$$
\sum_{k=0}^{+\infty}\ep^k\left(\sum_{\CP\colon |\CP|=k}\cont(\CP)\right)
$$
will be based on absolute convergence of the series
$$
\sum_{k=0}^{+\infty}\ep^k\left(\sum_{[\CP]\colon |\CP|=k}|\cont([\CP])|\right),
$$
where we assume that each equivalence class is counted once in the summation. It is easy to see (cf. Remark \ref{rem_summary3}) that the series for individual loops
\bee
\label{eq_divergent}
\sum_{k=0}^{+\infty}\ep^k\left(\sum_{\CP\colon |\CP|=k}|\cont(\CP)|\right)=\sum_{\CP}\ep^{|\CP|}|\cont(\CP)|.
\ene
will diverge if $\lambda_0$ is not an isolated eigenvalue of $H$ at $\ep=0$ and the hopping term is the usual Laplacian (if one considers more general hopping, one needs to assume that $\Gamma$ is not ``too disconnected'').
\end{rem}
\subsection{Some examples} The loop with canonical marking $\CL=(123[456[5]654]321)$ has four elements in its equivalence class: $(1234565654321)$, $(123(1232321)321)$, $(123(123(-1)321)321)$, and $(123456(-1)654321)$, graphically represented by
$$
\begin{tikzpicture}
	\begin{pgfonlayer}{nodelayer}
		\node [style=Vertex] (0) at (-1.5, 0) {$1$};
		\node [style=Vertex] (1) at (0.5, -1) {$2$};
		\node [style=Vertex] (2) at (2.75, -1.25) {$3$};
		\node [style=Vertex] (3) at (4.75, -0.25) {$4$};
		\node [style=Vertex] (4) at (6, 1.25) {$5$};
		\node [style=Vertex] (5) at (6.5, 3.25) {$6$};
		\node [style=Vertex] (6) at (6.25, 5.25) {$5$};
		\node [style=Vertex] (7) at (5.25, 6.75) {$6$};
		\node [style=Vertex] (8) at (3.75, 8) {$5$};
		\node [style=Vertex] (9) at (1.5, 8.25) {$4$};
		\node [style=Vertex] (10) at (-0.75, 7.5) {$3$};
		\node [style=Vertex] (11) at (-2.5, 6) {$2$};
		\node [style=Vertex] (12) at (-3, 4) {$1$};
		\node [style=Vertex] (13) at (-2.75, 1.75) {$0$};
		\node [style=Vertex] (14) at (5.25, -5) {$0$};
		\node [style=Vertex] (15) at (4.25, -3) {$1$};
		\node [style=Vertex] (16) at (5.5, -1.25) {$2$};
		\node [style=Vertex] (17) at (7.5, -1.25) {$3$};
		\node [style=Vertex] (18) at (8.75, -3) {$2$};
		\node [style=Vertex] (19) at (7.5, -5) {$1$};
		\node [style=base vertex] (20) at (8.75, 0.5) {$3$};
		\node [style=Vertex] (21) at (12.5, 4.25) {$2$};
		\node [style=Vertex] (22) at (8.75, 4.25) {$2$};
		\node [style=Vertex] (23) at (12.75, 0.5) {$2$};
		\node [style=Vertex] (24) at (8, 2.5) {$1$};
		\node [style=Vertex] (25) at (10.75, 5) {$3$};
		\node [style=Vertex] (26) at (13.25, 2.5) {$3$};
		\node [style=Vertex] (27) at (10.75, -0.25) {$1$};
	\end{pgfonlayer}
	\begin{pgfonlayer}{edgelayer}
		\draw [style=Arrow edge] (13) to (0);
		\draw [style=Arrow edge] (0) to (1);
		\draw [style=Arrow edge] (1) to (2);
		\draw [style=Arrow edge] (2) to (3);
		\draw [style=Arrow edge] (3) to (4);
		\draw [style=Arrow edge] (4) to (5);
		\draw [style=Arrow edge] (5) to (6);
		\draw [style=Arrow edge] (6) to (7);
		\draw [style=Arrow edge] (7) to (8);
		\draw [style=Arrow edge] (8) to (9);
		\draw [style=Arrow edge] (9) to (10);
		\draw [style=Arrow edge] (10) to (11);
		\draw [style=Arrow edge] (11) to (12);
		\draw [style=Arrow edge] (12) to (13);
		\draw (17) to (20);
		\draw [style=Arrow edge] (14) to (19);
		\draw [style=Arrow edge] (19) to (18);
		\draw [style=Arrow edge] (18) to (17);
		\draw [style=Arrow edge] (17) to (16);
		\draw [style=Arrow edge] (16) to (15);
		\draw [style=Arrow edge] (15) to (14);
		\draw [style=Arrow edge] (20) to (27);
		\draw [style=Arrow edge] (27) to (23);
		\draw [style=Arrow edge] (23) to (26);
		\draw [style=Arrow edge] (26) to (21);
		\draw [style=Arrow edge] (21) to (25);
		\draw [style=Arrow edge] (25) to (22);
		\draw [style=Arrow edge] (22) to (24);
		\draw [style=Arrow edge] (24) to (20);
	\end{pgfonlayer}
\end{tikzpicture}
$$
$$
\begin{tikzpicture}
	\begin{pgfonlayer}{nodelayer}
		\node [style=Vertex] (14) at (-3.5, -5.5) {$0$};
		\node [style=Vertex] (15) at (-4.5, -3.5) {$1$};
		\node [style=Vertex] (16) at (-3.25, -1.75) {$2$};
		\node [style=Vertex] (17) at (-1.25, -1.75) {$3$};
		\node [style=Vertex] (18) at (0, -3.5) {$2$};
		\node [style=Vertex] (19) at (-1.25, -5.5) {$1$};
		\node [style=base vertex] (20) at (0, 0) {$3$};
		\node [style=Vertex] (21) at (2, 0) {$1$};
		\node [style=Vertex] (22) at (3.25, 1.75) {$2$};
		\node [style=Vertex] (23) at (2.5, 3.5) {$3$};
		\node [style=Vertex] (24) at (0.25, 3.5) {$2$};
		\node [style=Vertex] (25) at (-0.75, 1.75) {$1$};
		\node [style=base vertex] (27) at (4.75, 3.75) {$3$};
		\node [style=Vertex] (28) at (7, 3.75) {};
		\node [style=Vertex] (29) at (7, 3.75) {};
		\node [style=Vertex] (30) at (7, 3.75) {$-1$};
		\node [style=Vertex] (31) at (6.5, -6.75) {$1$};
		\node [style=Vertex] (32) at (6.5, 1.25) {$5$};
		\node [style=Vertex] (33) at (2.5, -2.75) {$2$};
		\node [style=Vertex] (34) at (10.5, -2.75) {$4$};
		\node [style=Vertex] (35) at (8.75, 0.75) {$6$};
		\node [style=Vertex] (36) at (10.25, -1) {$5$};
		\node [style=Vertex] (37) at (10, -4.75) {$3$};
		\node [style=Vertex] (38) at (8.5, -6.25) {$2$};
		\node [style=Vertex] (40) at (3, -4.75) {$1$};
		\node [style=Vertex] (41) at (3, -0.5) {$3$};
		\node [style=Vertex] (42) at (4.5, 0.75) {$4$};
		\node [style=Vertex] (44) at (12.5, 2.25) {$-1$};
		\node [style=base vertex] (45) at (10.5, 2.25) {};
		\node [style=base vertex] (47) at (10.5, 2.25) {$6$};
		\node [style=Vertex] (48) at (4.5, -6.25) {$0$};
	\end{pgfonlayer}
	\begin{pgfonlayer}{edgelayer}
		\draw (17) to (20);
		\draw [style=Arrow edge] (14) to (19);
		\draw [style=Arrow edge] (19) to (18);
		\draw [style=Arrow edge] (18) to (17);
		\draw [style=Arrow edge] (17) to (16);
		\draw [style=Arrow edge] (16) to (15);
		\draw [style=Arrow edge] (15) to (14);
		\draw [style=Arrow edge, bend right, looseness=1.25] (27) to (30);
		\draw [style=Arrow edge, bend left=330, looseness=1.25] (30) to (27);
		\draw (23) to (27);
		\draw [style=Arrow edge] (20) to (21);
		\draw [style=Arrow edge] (21) to (22);
		\draw [style=Arrow edge] (22) to (23);
		\draw [style=Arrow edge] (23) to (24);
		\draw [style=Arrow edge] (24) to (25);
		\draw [style=Arrow edge] (25) to (20);
		\draw (35) to (47);
		\draw [style=Arrow edge, bend right] (47) to (44);
		\draw [style=Arrow edge, bend left=330] (44) to (47);
		\draw [style=Arrow edge] (48) to (31);
		\draw [style=Arrow edge] (31) to (38);
		\draw [style=Arrow edge] (38) to (37);
		\draw [style=Arrow edge] (37) to (34);
		\draw [style=Arrow edge] (34) to (36);
		\draw [style=Arrow edge] (36) to (35);
		\draw [style=Arrow edge] (35) to (32);
		\draw [style=Arrow edge] (32) to (42);
		\draw [style=Arrow edge] (42) to (41);
		\draw [style=Arrow edge] (41) to (33);
		\draw [style=Arrow edge] (33) to (40);
		\draw [style=Arrow edge] (40) to (48);
	\end{pgfonlayer}
\end{tikzpicture}
$$
\begin{rem}
\label{rem_summary3}
As mentioned in Remark \ref{rem_summary2}, the cancellation will be achieved inside each equivalence class. Before proceeding with the general construction, let us illustrate the usefulness of these equivalence classes on the following example. Let $d=1$, $V_{n}=f(n\omega)$, where, say, $f$ is a smooth $1$-periodic function on $\mathbb R$ such that and $f'(x)\ge 1$ in a neighborhood of $0$. Let $\omega\in \R\setminus\Q$ and $\Phi$ be the usual discrete Laplacian operator:
$$
\Phi_{ij}=(\Delta)_{ij}=\delta_{i,j+1}+\delta_{i,j-1}, \quad H=V+\varepsilon\Delta.
$$
Suppose that, say, $\dist(5\omega,\Z)\ll \varepsilon$, so that $|V_5|=|f(0)-f(5\omega)|\ll \varepsilon$. Suppose also that $5$ is the only small denominator: 
$$
\level(5)=1, \quad \level(1)=\level(2)=\level(3)=\level(4)=0,\quad \safedist(5)\ge 3.
$$
Typically, we would have $|V_0-V_5|\approx\varepsilon^{\safedist(5)}$, which would mean that one needs to make $\safedist(5)\ge 3$ steps to compensate for the factor $|f(0)-f(5\omega)|^{-1}$. Consider the following eigenvalue path:
$$
\CP_k=(123454545\ldots454321)
$$
where $\ldots$ represents repetition of "45" so that the loop has $k$ visits of $5$. Clearly,
$$
\cont(\CP_k)=(V_0-V_1)^{-1}(V_0-V_2)^{-1}(V_0-V_3)^{-1}(V_0-V_4)^{-(k+1)}(V_0-V_5)^{-k}(V_0-V_3)^{-1}(V_0-V_2)^{-1}(V_0-V_1)^{-1}.
$$
Since the loop $\CP_k$ only makes two steps between several consecutive visits of 5, it is considered non-safe, and its contributions to \eqref{eq_divergent} will blow up, even with the factor $\varepsilon^{|\CP_k|}$, as $k\to \infty$.
The canonical marking of $\CP$ will be as follows:
$$
\CP=(12345[4]5[4]5\ldots[4]54321),
$$
with the canonical translation
$$
T(\CP)=(12345(-1)5(-1)5\ldots(-1)54321).
$$
Here, the vertices in round brackets denote vertices on the sheet above $5$. After the translation, the vertex $4$ becomes $-1$. In both cases, there are $k-1$ times when a sheet of height $1$ is visited. The whole translation equivalence class consists of $2^{k-1}$ elements, since there are $k-1$ choices between $[4]$ and $(-1)$. One can factor the total contribution of the equivalence class as follows:
\begin{multline*}
\cont([\CP])=\cont(123454321)(V_0-V_5)^{-(k-1)}((V_0-V_4)^{-1}-(V_0-V_{-1})^{-1})^{k-1}\\
=\cont(123454321)\frac{((f(0)-f(4\omega))^{-1}-(f(0)-f(-\omega))^{-1})^{k-1}}{(f(0)-f(5\omega))^{k-1}},
\end{multline*}
since the change from $[4]$ to $(-1)$ replaces one of the terms $-V_{4}$ by $V_{-1}$. Note that, since $\|5\omega\|$ is small, we have $f(-\omega)\approx f(4\omega)$. Therefore,
$$
|(f(0)-f(4\omega))^{-1}-(f(0)-f(-\omega))^{-1}|\approx \|5\omega\| \frac{|f'(4\omega)|}{(f(0)-f(4\omega))^2}
$$
and
$$
|f(0)-f(5\omega)|\approx f'(0)\|5\omega\|\ge \|5\omega\|,
$$
which implies that
$$
|\cont(\CP)|\approx \cont(123454321)\frac{|f'(4\omega)|^{k-1}}{f'(0)^{k-1}|f(0)-f(4\omega)|^{2k-2}}.
$$
Since $f$ does not vanish around $4\omega$, we can see that the ``bad'' factor $\|5\omega\|^{-1}$ enters $\cont(\CP)$ with power $1$ rather than $k$.
\end{rem}
Unfortunately, applying this general idea to more complicated situations requires tedious bookkeeping, which we will now describe.
\subsection{Loop stacks} In order to describe the general case, motivated by the last example, we will use the following definition.
\begin{defn}
A {\it loop stack} is an eigenvalue/eigenvector path $\CP$ on $\Gamma$ satisfying the following:
\begin{enumerate}
	\item All loops of $\CP$ are safe.
	\item For any loop $\CL$ on $\CP$ that is attached to a vertex with coordinate $\bn$, we have $|\CL|<\safedist(\bn)$. In particular, the loops on $\CP$ can only be attached to small denominators.
\end{enumerate}
\end{defn}
We can now describe the translational equivalence class of any path $\CP$. Let $T(\CP)$ be the canonical translation of $\CP$. We will say that a loop $\CL$ on $T(\CP)$ is {\it short}, if it is attached to a vertex with coordinate $\bn$, with $|\CL|<\safedist(\bn)$. The base loop is never considered short. Since $T(\CP)\in [\CP]$, we can choose $T(\CP)$ as the ``natural representative'' of the equivalence class of $\CP$.
\begin{itemize}
\item
Let $\CP$ be an eigenvalue/eigenvector path on $\Gamma$, and suppose that $T(\CP)$ has $s$ short loops. Then $\#[\CP]=2^s$. The elements of $[\CP]$ are obtained by replacing some short loops on $T(\CP)$ by marked segments:
$$
\bm(\ldots)\bm \mapsto \bm[\ldots+\bm]\bm,
$$
where $(\ldots)$ denotes a short loop, and $\ldots+\bm$ denotes translation of all coordinates by $\bm$.
\item
All loops of $T(\CP)$ are safe, and $T(\CP)$ is the only element of $[\CP]$ with this property.
\item
One can introduce an attachment procedure for loop stacks similarly to loops: if $\CP_0$ is an eigenvalue/eigenvector path that has a vertex $v$ with coordinate $\bn$ and $\CP_1$ is an eigenvalue loop stack, one can attach $\CP_1$ to $\CP_0$ by replacing
$$
\bn\mapsto \bn\CP_1\bn
$$
in the string that describes $\CP_0$. Any path $\CP$ with $T(\CP)=\CP$ can be obtained from a base (eigenvalue/eigenvector) loop stack by a sequence of the above attachment operations, at the points of the base stack or already attached loop stacks. Each attached loop stack becomes a non-base loop stack on $\CP$.

Now, given a path $\CP=T(\CP)$, we would like to break it into loop stacks in some natural way. In general, there are multiple ways of doing so: for example, one can treat each loop as a separate loop stack. The opposite of this would be a decomposition with smallest possible number of loop stacks. Note that any stack may contain at most one non-short loop (which can only be its base loop). As a consequence, each non-short loop $\CL$ of $\CP$, including the base loop of $\CP$, must be the base loop of some stack. For each such loop, we can form the largest possible stack that starts from that loop using the following ``greedy algorithm'': consider all short loops attached to $\CL$; now, consider all short loops attached to the loops considered before, and repeat the process until we run out of attached short loops. It is easy to see that these stacks will not overlap and will ultimately contain each loop on $\CP$. We will call these stacks {\it maximal}.
\item
Suppose that the above procedure results in the path $\CP=T(\CP)$ being decomposed into maximal stacks $\CP_0,\ldots,\CP_n$, where $\CP_0$ contains the base loop of $\CP$, and $\CP_j$ is attached to a small denominator $\bn_j$ on one of the previous stacks $\CP_0,\ldots,\CP_{j-1}$. In particular, the base loops of $\CP_0,\CP_1,\ldots,\CP_n$ are all non-short loops on $\CP$. Then
\bee
\label{eq_end_sec2}
\cont([\CP])=\cont(\CP_0)\prod_{s=1}^n 
(V_{\bn_s}-V_{\bze})^{-1}\cont([\CP_s]).
\ene
\end{itemize}
Our next objective is to obtain a bound of the form $\ep^{|\CP_s|}|\cont([\CP_s])|\lesssim \ep^{c|\CP_s|}$ by taking advantage of cancellations inside $[\CP_s]$. In other words, we expect $[\CP_s]$ to behave as a safe loop of comparable length, thanks to the base loop of $\CP_s$ being non-short. Since $|\CP_s|\ge \safedist(\bn_s)$, we can absorb the $(V_{\bn_s}-V_{\bze})^{-1}$ factors into $\cont([\CP_s])$ if we have bounds of the form
$$
\ep^{c\cdot  \safedist(\bn_s)}\ll |V_{\bze}-V_{\bn_s}|.
$$
In the quasiperiodic case, this will be achieved by imposing a Diophantine condition on the frequency, combined with choosing a function $\safedist$ that does not grow too slowly, see Section 5 and, in particular, Theorem \ref{th_loop_stack}.
\section{Convergence of the perturbation series}
\subsection{Regular points of monotone potentials}
The cancellation procedure described above will be applied to quasiperiodic operators, that is, when
$$
V_{\bn}=f(x+\bn\cdot\omega),
$$
$f$ is an $1$-periodic function on $\mathbb R$, and $\omega=(\omega_1,\ldots,\omega_d)$ is a vector with rationally independent components. Theorem \ref{th_lak} gives an asymptotic series for an eigenvalue which would have been an analytic continuation of the eigenvalue $V_{\bze}=f(x)$ of the operator at $\varepsilon=0$, if that eigenvalue was isolated. In our case, it is not isolated; however, as long as the algebraic non-resonance condition
$$
f(x+\bn\cdot\omega)\neq f(x)\quad \text{for}\,\,\bn\neq\bze
$$
is satisfied, each term of the perturbation series is well-defined as discussed in Section 3.

We will formulate the convergence results locally, that is, for fixed $x$, assuming that $f$ is monotone and one-to-one in a neighborhood of $x$. We will also provide sufficient conditions for obtaining convergent expansions for eigenvectors for all $x\in \R\setminus(\Z+1/2+\Z^d\cdot\omega)$. In all cases, we will consider the eigenvector associated with the origin (that is, a perturbation series that starts from $\psi_0=e_{\bze}$, $\lambda_0=f(x)$). However, if we obtain such eigenfunctions for all $H(x)$ with $x\in \{x_0+\bn_0\cdot \omega:\bn_0\in \Z^d\}$, their translations will form a complete system of eigenvectors of $H(x_0)$ sufficient to establish Anderson localization for that operator.

We will always assume the following: 
\begin{itemize}
	\item[(f1)] $f\colon (-1/2,1/2)\to \R$ is continuous, $\quad f(0+0)=-\infty,\quad f(1-0)=+\infty$, and is extended by $1$-periodicity into $\mathbb R\setminus(\Z+1/2)$.
\end{itemize}
Suppose, $f$ satisfies $\mathrm{(f1)}$. Let 
$$
C_{\mathrm{reg}}>0,\quad  x_0\in (-1/2,1/2)
$$
We say that $f$ is $C_{\mathrm{reg}}$-regular at $x_0$, if 
\begin{itemize}
\item[(cr0)]The pre-image $f^{-1}((f(x_0)-2,f(x_0)+2))\cap (-1/2,1/2)$ is an open interval (denoted by $(a,b)$), and $\left.f\right|_{(a,b)}$ is a one-to-one map between $(a,b)$ and $(f(x_0)-2,f(x_0)+2)$.
\item[(cr1)]Let $D_{\min}(x_0):=\inf\limits_{x\in (a,b)}f'(x)$. Then,
\bee
\label{eq_cr1}
D_{\min}(x_0)\le f'(x)\le C_{\mathrm{reg}} D_{\min}(x_0),\quad \forall x\in (a,b).
\ene
At the points where $f'(x)$ does not exist, we require that the inequalities hold for all derivative numbers (the $\inf$ is also taken over the set of all derivative numbers).
\item[(cr2)]Define $(a_1,b_1)=f^{-1}(f(x_0)-1,f(x_0)+1)\subset (a,b)$, and

$$
g(x)=\frac{1}{f(x_0)-f(x)}, \quad x\in (b_1,a_1+1),
$$
extended by continuity to $g(0)=g(1)=0$ (recall that we also assume $f(x+1)=f(x)$, so that the interval $(b_1,a_1+1)$ is essentially $(-1/2,1/2)
\setminus(a_1,b_1)$ together with the point $0=1\,\,\mathrm{mod}\,\,1$). Then, under the same conventions on the existence of derivatives,
$$
|g'(x)|\le C_{\mathrm{reg}} D_{\min}(x_0),\quad x\in (b_1,a_1+1).
$$
\end{itemize}
\begin{rem}
Suppose, $f$ satisfies (f1) and 
$$
f'(x)<C f'(y) (1+|f(x)-f(y)|^2)\quad \text{for all}\,\,x,y\in (-1/2,1/2).
$$
Then $f$ is $C_{\mathrm{reg}}$-regular on $(-1/2,1/2)$ (with the value of $C_{\reg}$ depending on $C$). In particular, any meromorphic function satisfying $\mathrm{(f1)}$, such that $f'(x)>0$ on $(-1/2,1/2)$, has this property.
\end{rem}
\begin{rem}
One can see Condition (cr1) as a statement that $f'$ does not oscillate too much on intervals of length determined by the change of the values of $f$. Note that the interval $(a,b)$ will shrink as one of its endpoints approaches $\{-1/2,1/2\}$. Condition (cr2) is a ``regularity at infinity'' condition, which will be important to account for non-small denominators.
\end{rem}
\begin{rem}
\label{rem_creg}
In some later results, it will be convenient to consider a ``rescaled'' version of $C_{\reg}$-regularity.  To obtain the definition of $C_{\reg,\nu}$-regularity, replace the interval $(f(x_0)-2,f(x_0)+2)$ in (cr0) by $(f(x_0)-2\nu,f(x_0)+2\nu)$ and the interval $(f(x_0)-1,f(x_0)+1)$ in (cr2) by $(f(x_0)-\nu,f(x_0)+\nu)$.
\end{rem}
%\subsection{Outline}
%It was observed in Section 2 that the the perturbation series for the eigenvalue has the following form:
%\label{eq_new_series}
%\lambda(\varepsilon)=\sum_{j=0}^{\infty} \varepsilon^j \sum\limits_{|\CP|=j}C(\CP),
%\ene
%and the last sum is a sum of at most $c(\varphi)^j$ terms of the form $C([\CP_1])C([\CP_2])\ldots C([\CP_k])$, where 

%It was observed in Section 2 that, in order to establish convergence of perturbation series, it is sufficient to establish a bound
%\bee
%\label{eq_single_multiplicative}
%|C([\CP])|\le (C_1\varepsilon)^{-C_2|\CP|},
%\ene
%where $C_1\varepsilon\ll 1$, $C_2\ll 1$, for all refined loop configurations $\CP$ that are obtained only using rules (1) -- (3) of Definition \ref{def_refined0}. Indeed, \eqref{eq_multiplicative} implies that the contribution of any refined loop configuration is a product of factors \eqref{eq_single_multiplicative}. The total number of configurations of length $j$ is bounded by the combinatorial factor $c(\varphi)^{j}$. Hence, both the product of \eqref{eq_single_multiplicative} and the combinatorial factor will be dominated by $\varepsilon^j$, and together will be bounded by a term of a convergent geometric series. In this section, we obtain estimates for $C([\CP_j])$.

\subsection{Consistent level functions for monotone quasiperiodic potentials}
We will start by constructing a consistent denominator data on $\Z^d$, suitable for use with locally monotone quasiperiodic operators. Recall that we need to define two functions
$\safedist$ and $\level$, satisfying
\begin{itemize}
	\item[(c0)] $\level(\bze)=+\infty$, $\safedist(0)=0$. The function $\safedist$ is monotone non-decreasing in its argument.
	\item[(c1)] $\dist_{\varphi}(\bm,\bn)\ge \min\{\safedist(\bm),\safedist(\bn)\}$, for $\bm\neq \bn$.
	\item[(c2)] Suppose $0<\dist_{\varphi}(\bn,\bm)<\safedist(\bm)$. Then $\level(\bn-\bm)=\level(\bn)$.
\end{itemize}
Let $\omega\in (-1/2,1/2)^d$ be a Diophantine frequency vector, that is,
\bee
\label{eq_diophantine}
\|\bn \cdot\omega\|=\dist(\bn\cdot\omega,\Z)\ge C_{\mathrm{dio}}|\bn|^{-\tau},\quad \forall \bn\in \Z^d\setminus\{\bze\}.
\ene
%\bee
%\label{eq_diophantine2}
%\|\bn \cdot\omega^2\|=\dist(\bn\cdot\omega^2,\Z^d)\ge C_{\mathrm{dio}}|\bn|^{-\tau},
%\ene
%where $\omega^2=(\omega_1^2,\ldots,\omega_d^2)$. The latter condition implies
%$$
%\dist (\|\bn\cdot\omega\|,\Z)\ge C_{\mathrm{dio}}|\bn|^{-\tau}.
%$$
We will always assume that some $\tau>1$ is fixed.
\begin{thm}
\label{th_consistent_diophantine}
Fix $C_{\mathrm{safe}}>0$ and, for $\beta>0$, define
\bee
\label{eq_beta_def}
\beta_k:=\lfloor \beta^{-k-1}\rfloor^{-1},\quad k\in \Z_+;\quad \beta_{-1}:=+\infty.
\ene
Define the level and safe distance functions as follows:
\bee
\label{eq_leveling}
\level(\bn)=k\quad \text{iff} \quad \beta_{k}\le \|\bn\cdot \omega\|<\beta_{k-1},\quad \bn\in \Z^d\setminus\{\bze\};\quad \level(\bze)=+\infty;
\ene
\bee
\label{eq_safe_distance}
\safedist(\bn)=\lceil C_{\mathrm{safe}} \level(\bn)^3\rceil,\quad \bn\in \Z^d.
\ene
Then, for 
\bee
\label{eq_epsilon_safe}
0<\beta<\beta_{\max}(C_{\mathrm{safe}},C_{\mathrm{dio}},\dist_{\varphi},\tau)
\ene
the constructed functions $\level(\bn)$ and $\safedist(\bn)$ satisfy $(c1)$ and $(c2)$, and therefore define a consistent denominator data on $\Z^d$.\end{thm}
\begin{proof}
Suppose, $\min\{\level(\bn),\level(\bm)\}=k\ge 1$. The Diophantine condition implies that, if $\bn\neq \bm$,
\bee
\label{eq_safedist_exponential}
\dist_{\varphi}(\bm,\bn)\ge c_1(C_{\mathrm{dio}},\tau,\dist_{\varphi})\beta^{-k/\tau}.
\ene
One can obtain (c1) as long as $\beta_{\max}$ is chosen to satisfy
$$
C_{\mathrm{safe}} k^3 \le c_1(C_{\mathrm{dio}},\tau,\dist_{\varphi})\beta^{-k/\tau},
$$
that is,
$$
\beta\le c_2(C_{\mathrm{dio}},\tau,\dist_{\varphi})^{\tau/k}C_{\mathrm{safe}}^{-\tau/k}k^{-3\tau/k},\quad \forall k\ge 1.
$$
The worst case is, essentially, either $k=1$ or $k\to\infty$ (up to a factor that can be absorbed into $c_2$), and hence $\eqref{eq_epsilon_safe}$ implies (c1), under an appropriate choice of $\beta_{\max}(C_{\mathrm{safe}},C_{\mathrm{dio}},\dist_{\varphi},\tau)$.

Let us establish (c2). Let $\bm,\bn\in \Z^d$. By the Diophantine condition, for $k\ge 0$,
\bee
\label{eq_lower_dio}
|\|(\bn-\bm)\cdot\omega\|-\beta_k|\ge \lfloor  \beta^{-k-1}\rfloor^{-1}\|\lfloor \beta^{-k-1}\rfloor(\bn-\bm)\cdot\omega\|\ge C_{\mathrm{dio}}\beta_k^{1+\tau}|\bn-\bm|^{-\tau}.
\ene
If 
\bee
\label{eq_if47}
C_{\mathrm{dio}}|\bn-\bm|^{-\tau}\ge 2\beta_k,
\ene
then one can also obtain a better bound
\bee
\label{eq_lower_better}
|\|(\bn-\bm)\cdot\omega\|-\beta_k|\ge \|(\bn-\bm)\cdot\omega\|-\beta_k \ge \frac12 C_{\mathrm{dio}}|\bn-\bm|^{-\tau}.
\ene
If \eqref{eq_lower_better} cannot be obtained, then the opposite of \eqref{eq_if47} holds, that is,
\bee
\label{eq_lower_alternative}
|\bn-\bm|>2^{-1/\tau} C_{\mathrm{dio}}^{1/\tau} \beta_k^{-1/\tau}.
\ene
Now, suppose that $0<\dist_{\varphi}(\bn,\bm)<\safedist(\bm)$. Then \eqref{eq_safe_distance} yields
$$
\level(\bm)\ge C_{\mathrm{safe}}^{-1/3}\dist_{\varphi}(\bn,\bm)^{1/3}\ge C_{\mathrm{safe}}^{-1/3}c(\dist_{\varphi})|\bn-\bm|^{1/3}>0,
$$
which, combined with \eqref{eq_leveling}, implies
\bee
\label{eq_momega_bound}
\|\bm\cdot\omega\|\le 2\beta^{\lceil C_{\mathrm{safe}}^{-1/3}c(\dist_{\varphi})|\bn-\bm|^{1/3}\rceil}.
\ene
We would like to show that addition of $\bm$ will not change the level of $\bn-\bm$. This would follow from
$$
\|\bm\cdot\omega\|<\min\{\|(\bn-\bm)\cdot\omega\|-\beta_k,\beta_{k-1}-\|(\bn-\bm)\cdot\omega\|\},\quad \text{where}\quad k=\level(\bn-\bm).
$$
We will verify that the above inequality follows from the lower bounds \eqref{eq_lower_dio}, \eqref{eq_lower_better} and the upper bound \eqref{eq_momega_bound}. Suppose, \eqref{eq_lower_better} holds. Then we need
\bee
\label{eq_eps_gamma0_1}
2\beta^{\lceil C_{\mathrm{safe}}^{-1/3}c(\dist_{\varphi})|\bn-\bm|^{1/3}\rceil}<\frac12 C_{\mathrm{dio}}|\bn-\bm|^{-\tau}.
\ene
Suppose now that \eqref{eq_lower_better} does not hold and we have to use \eqref{eq_lower_dio}. Then one needs to establish
\bee
\label{eq_eps_gamma0_2}
2\beta^{\lceil C_{\mathrm{safe}}^{-1/3}c(\dist_{\varphi})|\bn-\bm|^{1/3}\rceil}<C_{\mathrm{dio}}\beta_k^{1+\tau}|\bn-\bm|^{-\tau}.
\ene
Using \eqref{eq_lower_alternative}, one can replace $|\bn-\bm|$ in the left hand side by 
$$
\delta_{k,k}:=\frac12 |\bn-\bm|+2^{-1-1/\tau}C_{\mathrm{dio}}^{1/\tau}\beta_k^{-1/\tau}
$$ and reduce to the inequality
\bee
\label{eq_412_reduced}
2\beta^{\lceil C_{\mathrm{safe}}^{-1/3}c(\dist_{\varphi})\delta_{k,k}^{1/3}\rceil}<C_{\mathrm{dio}}\beta_k^{1+\tau}|\bn-\bm|^{-\tau}.
\ene
Both \eqref{eq_eps_gamma0_1} and \eqref{eq_412_reduced} can be obtained by taking a sufficiently small $\beta$, as the bound in the left hand side is exponentially better both in $|\bn-\bm|$ and in $\beta_k$. If $k\ge 1$, one also needs a similar bound for $|\|(\bn-\bm)\cdot\omega\|-\beta_{k-1}|$, which reduces to
$$
2\beta^{\lceil C_{\mathrm{safe}}^{-1/3}c(\dist_{\varphi})|\bn-\bm|^{1/3}\rceil}<\frac12 C_{\mathrm{dio}}|\bn-\bm|^{-\tau}
$$
if $C_{\mathrm{dio}}|\bn-\bm|^{-\tau}\ge 2\beta_{k-1}$, and otherwise to
$$
2\beta^{\lceil C_{\mathrm{safe}}^{-1/3}c(\dist_{\varphi})\delta_{k,k-1}^{1/3}\rceil}<C_{\mathrm{dio}}\beta_{k-1}^{1+\tau}|\bn-\bm|^{-\tau},
$$
where 
$$
\delta_{k,k-1}:=\frac12 |\bn-\bm|+2^{-1-1/\tau}C_{\mathrm{dio}}^{1/\tau}\beta_{k-1}^{-1/\tau}.
$$
In both cases, the inequalities can be obtained in the same way as before.
\end{proof}

%\begin{rem}
%\label{rem_leveling}
%Suppose that $f$ is $(C_{\mathrm{reg}},r)$-regular at $x_0$. This implies
%$$
%|V_{\bn}|=|f(x_0+\bn\cdot\omega)-f(x_0)|\ge \min(r,C_{\mathrm{reg}}^{-1}f'(x_0)\|\bn\cdot\omega\|),
%$$
%and hence
%$$
%\bn\in \CD_k^- \quad \text{implies} \quad |V_{\bn}| \ge \min\{r,C_{\mathrm{reg}}^{-1}\varepsilon^{k+\delta_{\mathrm{lev}}}\};
%$$
%$$
%|V_{\bn}|> C_{\mathrm{reg}} \varepsilon^{k+\delta_{\mathrm{lev}}}  \quad  \text{implies} \quad \bn\in \CD_k^-.
%$$
%In particular, 
%$$
%\bn \in \CD_0 \quad  \text{implies} \quad |V_{\bn}|\ge \min\{r,C_{\mathrm{reg}}^{-1}\varepsilon^{\delta_{\mathrm{lev}}}\}.
%$$
	
%\end{rem}
\subsection{Lipschitz bounds for loop stacks} In this subsection, we will obtain the central inductive estimates on $\cont([\CP])$, where $\CP$ is a loop stack, by taking advantage of the cancellations. It will be convenient to introduce some notation in the beginning. In the remainder of the section, we will assume that $\beta\in (0,1)$ is fixed and the denominators are leveled according to Theorem \ref{th_consistent_diophantine}. The main results will be obtained under additional assumptions on $\beta$ being small enough. Let $\CP$ be an eigenvalue/eigenvector loop stack.
\begin{itemize}
\item 
Denote by $\height(\CP)$ the height of $\CP$, which is the maximal height of sheets of $\Gamma$ visited by $\CP$.
\item Let $\height(k)$ and $\maxlevel(k)$, respectively, be the maximal possible value of $\height(\CP)$ and the maximal level of a denominator on $\CP$, 
considered over all loop stacks $\CP$ with $|\CP|=k$.
\item
Let $m\in \mathbb R$. For each loop of $\CP$, calculate the maximal level of denominators on that loop. Among these numbers, consider only those that are $\ge m$ and add them together. Denote the resulting number by $\level(\CP,m)$. Let also $\level(\CP)=\level(\CP,0)$.
\item Let $\denominators(\CP,m)$ be the total number of denominators on $\CP$ of level $\ge m$. 
Here, each denominator is counted as many times as the corresponding lattice point is visited by $\CP$; however, the contributions from the descending edges are not counted. Let also $\totallevel(\CP,m)$ denote the sum of levels of all these denominators.
%\item
%Let $\maxlevel(k)$ be the maximal possible level of a single denominator that can appear on a small denominator configuration of size $k$. Note that $\maxlevel(k)$ is attained on some configuration which consists of a single safe loop.
\item Let $\loops(\CP,m)$ be the number of loops of $\CP$ that contain a denominator of level $\ge m$. We have $\den(\CP,0)+\loops(\CP,0)=|\CP|$.
\item Let $\nbloops(\CP,m)$ and $\nblevel(\CP,m)$ denote the same quantities as above ($\loops$ and $\level$), but the base loop is not counted.
\item Let also $\downedges(\CP,m)$ be the number of descending edges on $\CP$ that lead to denominators of level $\ge m$. We have $\downedges(\CP,0)=\loops(\CP,0)-1$, since $\CP$ has to exit each non-base loop once.
\end{itemize}
\begin{lem}
\label{lemma_l_h_bounds}
In the above notation, the following bounds hold:
\bee
\label{eq_l_bound}
\maxlevel(k)\le \max\left\{ c(C_{\mathrm{dio}},\tau,\dist_{\varphi})\frac{\log(k)}{\log(\beta^{-1})},0\right\}
\ene
\bee
\label{eq_h_bound}
\height(k)\le 1+\height(\safedist(\maxlevel(k)))=1+\height(\lceil C_{\mathrm{safe}}\maxlevel(k)^3\rceil).
\ene
\end{lem}
\begin{proof}
Let $\CP$ be an eigenvalue/eigenvector loop stack with $|\CP|=k$. Estimate \eqref{eq_l_bound} follows directly from the Diophantine property. To show \eqref{eq_h_bound}, note that any loop on $\CP$ that is directly attached to the base loop, can be of length at most $\safedist(\maxlevel(k))$. Then one can consider the stack that starts from that loop and apply the induction assumption.
\end{proof}
\begin{rem}
\label{rem_multilog}
In order to obtain a meaningful bound, one needs to choose the parameters to satisfy $C_{\mathrm{safe}}\maxlevel (k)^3\ll k$, in which case $\height(k)$ will be a very slowly growing function. The bound can be achieved either by choosing small $C_{\mathrm{safe}}$ or small $\beta$.
\end{rem}

Recall that, for a fixed $x_0\in \R\setminus(\Z+1/2+\Z^d\cdot\omega)$, we have
$$
V_{\bn}=f(x_0+\bn\cdot\omega).
$$
In order to describe loop translations, it will be convenient to introduce an additional parameter $t$ as follows:
$$
V_{\bn}(t):=\begin{cases}
f(t+x_0+\bn\cdot\omega),&\bn\neq \bze\\
f(x_0),&\bn=\bze.
\end{cases}
$$
so that $V_{\bn}(0)=V_{\bn}$, $V_{\bze}(t)\equiv V_{\bze}$. Denote by $\cont(\CP,t)$ the function obtained by replacing all $V_{\bn}$ in $\cont(\CP)$ with $V_{\bn}(t)$. We will only consider this function in a neighborhood of the origin small enough (cf. \eqref{eq_smallenough}) so that $V_{\bn}(t)\neq V_{\bze}$ for $\bn\neq\bze$, thus avoiding zero denominators. Clearly, $\cont(\CP)=\cont(\CP,0)$. 
Let us also extend the above definition to the equivalence class:
$$
\cont([\CP],t):=\sum\limits_{\CP'\in [\CP]} \cont(\CP',t).
$$
%\item Let $\difflevel(\CP,\mbe)$ be the following quantity. For each loop $\CL$ that is attached to denominator of level $\ge m$, add $\level(\CL)$ if it's $\ge m$, and subtract it if it's $<m$.
% Let also
%$$
%\height(k)=\max\limits_{|\CP|=k}\height(\CP),
%$$
%the maximal possible height of a refined loop configuration of length $k$.

%\item Similarly to the previous definition, let $l_{\mathrm{att}}(\CP)$ be the sum of levels of denominators all attachment points of $\CP$, where each attachment is counted separately. For example, if a denominator $\bn_j$ has three loops attached to it, its level is counted three times.
\begin{lem}
\label{lemma_single_denominator}
%Suppose that $0<\varepsilon^{\delta_{\mathrm{lev}}}\le C_{\reg}$.
Suppose, $f$ is $C_{\mathrm{reg}}$-regular at $x_0$. 
Let $\bn \in \Z^d\setminus\{0\}$, $|t|\le \frac14 \|\bn\cdot\omega\|$.
Then
\begin{enumerate}
	\item $\left|\frac{1}{V_{\bze}-V_{\bn}(t)}\right|\le \max\left\{\frac{4}{D_{\min}(x_0)\beta^{\level(\bn)+1}},1\right\}$;
	\item $\left|\frac{V_{\bn}'(t)}{(V_{\bze}-V_{\bn}(t))^2}\right|\le C_{\reg}\cdot\max\left\{\frac{4}{\beta^{\level(\bn)+1}},D_{\min}(x_0)\right\}\cdot(\text{r.h.s of }(1))$.
\end{enumerate}
\end{lem}
\begin{proof}
First, let us note that, if $\{t+\bn\cdot\omega\}$ is large enough so that 
$$
x_0+t+\bn\cdot\omega\notin (a,b):=f^{-1}((f(x_0)-2,f(x_0)+2))\cap (-1/2,1/2),
$$
then one can use (cr0) and bound the denominator by $1$ in (1) and use (cr2) in (2). Otherwise, one can use the lower bound in (cr1) \eqref{eq_cr1} combined with $\|t+\bn\cdot\omega\|\ge \frac34 \|\bn\cdot\omega\|$.
\end{proof}
%\begin{rem}
%Suppose that $\varepsilon^{\dlev}\le \frac{D_{\min}}{4\gamma}$. Then all small denominators will have $\max\{\cdot,\cdot\}$ in Lemma \ref{lemma_single_denominator} attained on the first argument, and all non-small denominators will have it attained on the second argument.
%\end{rem}
If $D_{\min}(x_0)$ is large, some denominators of small levels may actually be not very small. It is convenient to introduce an extra parameter which will indicate the minimal level at which a denominator can in principle be smaller than $1$. Denote by
\bee
\label{eq_m_definition}
M_{\beta}=\frac{\log (4/D_{\min}(x_0))}{\log\beta}-1,
\ene
that is,
$$
\beta^{M_{\beta}+1}=\frac{4}{D_{\min}(x_0)}.
$$
In the sequel, it will be convenient to use the notation $\bn\in \CL$ if the loop $\CL$ visits a lattice point $\bn$.
\begin{lem}
\label{lemma_safe_loop_contrib}
Let $\CL$ be an eigenvalue/eigenvector loop. Suppose, $f$ is $C_{\mathrm{reg}}$-regular at $x_0$ and 
\bee
\label{eq_smallenough}
\quad |t|\le \frac14 \min_{\bn\in \CL}\|\bn\cdot\omega\|.
\ene
%Denote
%$$
%M=\frac{\log(4\gamma)-\log D_{\min}}{\log\varepsilon}-\dlev.
%$$
Define $M_{\beta}$ by $\eqref{eq_m_definition}$.
Then
\begin{enumerate}
	\item $|\cont(\CL,t)|\le \|\varphi\|_{\infty}^{|\CL|} \left(\frac{4}{D_{\min}(x_0)}\right)^{\denominators(\CL,M_{\beta})} \beta^{-\totallevel(\CL,\mbe)-\denominators(\CL,\mbe)}$.
	\item The function $\cont(\CL,\cdot)$ is Lipschitz continuous:
$$
\left|\frac{d}{dt}\cont({\CL,t})\right|\le|\CL|\cdot C_{\reg}\max\left\{\frac{4}{\beta^{\level(\CL)+1}},D_{\min}(x_0)\right\}\cdot(\text{r.h.s of }(1)).
$$
\end{enumerate}
\end{lem}
\begin{proof}
Both estimates directly follow from Lemma \ref{lemma_single_denominator} and counting the contributions from each denominator.
\end{proof}
\begin{rem}
The previous lemma does not assume $\CL$ to be safe. If it is safe, then one can estimate the number of denominators of level $\ge M\ge 1$ on $\CL$ by $\lceil c(\dist_{\varphi})\frac{|\CL|}{M^3 C_{\mathrm{safe}}}\rceil$. As a consequence, 
$
\totallevel(\CL,0)\le \frac{C_{\dist}}{C_{\safe}}|\CL|.$ Here, $C_{\dist}$ denotes a constant that only depends on the distance function.
\end{rem}

The following theorem is the main technical result of the section. It establishes a bound on the class of equivalence of a single loop stack, thus providing a bound on each $\cont([\CP_s])$ in \eqref{eq_end_sec2}.
\begin{thm}
\label{th_loop_stack}
Let $\CP$ be an eigenvalue/eigenvector loop stack. Suppose, $f$ is $C_{\mathrm{reg}}$-regular at $x_0$ and $\CL$ is the base loop of $\CP$. Define $\mbe$ by $\eqref{eq_m_definition}$. Let
$$
|t|\le \frac14 \min_{\bn\in \CL}\|\bn\cdot\omega\|,\quad 0<\beta< \beta_{\max}(\dist_{\varphi},C_{\dio},\tau,C_{\safe}).
$$
Then
\begin{enumerate}
	\item $$|\cont([\CP],t)|\le (C_{\dist}\|\varphi\|_{\infty})^{|\CP|} 
	\left(\frac{4}{D_{\min}(x_0)}\right)^{\denominators(\CP,\mbe)} 
\beta^{-\totallevel(\CP,\mbe)-\denominators(\CP,\mbe)}\times
$$
$$
\times C_{\reg}^{\downedges(\CP,\mbe)}\left(\frac{4}{D_{\min}(x_0)}\right)^{\nbloops(\CP,\mbe)}\beta^{-\nblevel(\CP,\mbe)-\nbloops(\CP,\mbe)}.
$$
\item The function $\cont(\CP,\cdot)$ is Lipschitz continuous with the bound on the derivative
$$\left|\frac{d}{dt}\cont([\CP],t)\right|\le C_{\reg}\cdot \max\left\{\frac{4}{\beta^{\level(\CL)+1}},D_{\min}(x_0)\right\}\cdot(\text{r.h.s of }(1)).
$$
\end{enumerate}
\end{thm}
\begin{proof}
Before starting the proof, we have a few remarks on the structure of the terms. The first line in the right hand side of (1) is the ``safe'' part of the contribution, similarly to (1) in Lemma \ref{lemma_safe_loop_contrib}. This is what the estimate would be like if we ignored all extra factors appearing in the attachment points. As in the example following Definition \ref{def_trans_eq}, the extra attachment factors will be cancelled with small numerators that appear as the result of subtracting a loop contribution term and its translation. Each such cancellation will introduce a derivative factor, similar to the extra factor (2) in Lemma \ref{lemma_safe_loop_contrib}. We only need to do this if the attachment denominator $V_{\bn_j}(t)-V_{\bze}$ is small enough; that is, $\level(\bn_j)\ge \mbe$ or, equivalently, that we cannot bound $|V_{\bn_j}(t)-V_{\bze}|$ by $1$ from below. The number of such attachments is equal to $\downedges(\CP,\mbe)$. For each of these attachments, we apply the induction assumption (2), which ultimately reduces to part (2) of Lemma \ref{lemma_safe_loop_contrib}. In this case, there are two possibilities: the level of the attached loop is still $\ge \mbe$, or it becomes less than $\mbe$ (note that it must be smaller than $\level(\bn_j)$ due to the shortness condition in the loop stack structure). In both cases, we get $C_{\reg}$ in the derivative bound, but the rest will depend on which argument of the $\max$ function is bigger. The number of times we use the first factor is, ultimately, controlled by $\nblevel$ and $\nbloops$.

An important observation is that, in estimating the derivative factor, one can always use the loop at the lowest level, since it will always contain the smallest denominator. The contribution from differentiating other loops is additive, and can hence be absorbed into a combinatorial factor $C_{\dist}^{|\CP|}$, where $C_{\dist}$ only depends on the distance function.

Both claims are proved by induction in $\height(\CP)$ at the same time. The case $\height(\CP)=0$ is contained in Lemma \ref{lemma_safe_loop_contrib}. Any $\CP$ satisfying the assumptions of the theorem can be split into the base loop $\CL$ (which is eigenvalue/eigenvector loop, depending on the type of $\CP$) and several eigenvalue loop stacks $\CP_j$, where $\CP_j$ is attached to a vertex $v_j$ on $\CL$ with coordinate $\bn_j$, which is a small denominator. Then, since the loop translation operation is equivalent to replacing $t$ by $t+\bn_j\cdot\omega$, we have
\bee
\label{eq_prod}
\cont([\CP],t)=\cont(\CL,t)\prod_j(V_{\bze}-V_{\bn_j}(t))^{-1}(\cont([\CP_j],t+\bn_j\cdot\omega)-\cont([\CP_j],t)).
\ene
Clearly, 
$$
\level(\CP,\mbe)=\level(\CL,\mbe)+\sum_j \level(\CP_j,\mbe),\quad |\CP|=|\CL|+\sum_j(1+|\CP_j|),
$$
$$
\loops(\CP)=\loops(\CL)+\sum_j \loops(\CP_j),
$$
$$
\den(\CP,\mbe)=\den(\CL,\mbe)+\sum_j \den(\CP_j,\mbe).
$$
$$
\totallevel(\CP,\mbe)=\totallevel(\CL,\mbe)+\sum_j \totallevel(\CP_j,\mbe).
$$
For each $\bn_j$, calculate the estimate of $|V_{\bze}(x)-V_{\bn_j}(x)|^{-1}$ provided by Lemma \ref{lemma_single_denominator}. The cases will be based on which of the arguments of the $\max$ function in that lemma is larger. If the lemma provides $|V_{\bze}-V_{\bn_j}(t)|^{-1}\le 1$ (the second argument of $\max$), use the direct estimate
\bee
\label{eq_simple_induction}
|\cont([\CP_j],t+\bn_j\cdot\omega)-\cont([\CP_j],t)|\le |\cont([\CP_j],t+\bn_j\cdot\omega)|+|\cont([\CP_j],t)|
\ene
and the induction assumption (1) (we will verify the necessary prerequisite shortly). This means that $\beta$ is small enough so that the denominator $(V_{\bze}-V_{\bn_j}(t))^{-1}$ does not need to be treated as small, and no derivative factors appear. Each time the estimate is applied, we gain the bound from the previous induction step and an extra factor of $2$, which can be absorbed into the combinatorial factor.

In the case $|V_{\bze}-V_{\bn_j}(t)|$ is smaller (meaning that the best that Lemma \ref{lemma_single_denominator} can provide is $|V_{\bze}-V_{\bn_j}(t)|^{-1}\le B$ with some $B>1$, using the first argument of the $\max$ function), we apply the induction assumption (2) and estimate $|\cont([\CP_j],t+\bn_j\cdot\omega)-\cont([\CP_j],t)|$ using the derivative of $\CP_j$.

Let us verify the prerequisites for applying the induction assumptions on the range of $t$. Since $|t|\le \frac14 \|\bn_j\cdot\omega\|$, the assumption would follow from
%\bee
%\label{eq_appropriate_bounds}
%|x|\le \frac{\varepsilon^{\level(\CL_j)+\delta_{\mathrm{lev}}}}{8\gamma},\quad \|\bn_j\cdot\omega\|\le \frac{\varepsilon^{\level(\CL_j)+\delta_{\mathrm{lev}}}}{8\gamma},
%\ene
\bee
\label{eq_appropriate_bounds}
\|\bn_j\cdot\omega\|\le \frac18 \min_{\bn\in \CL_j}\|\bn\cdot\omega\|,
\ene
where $\CL_j$ is the base loop of $\CP_j$. Recall that the denominator leveling constructed in Theorem \ref{th_consistent_diophantine} implies
\bee
\label{eq_appropriate_2}
\|\bn_j\cdot\omega\|\le  2\beta^{\level(\bn_j)},\quad \|\bn\cdot\omega\|\ge  \beta^{\level(\CL_j)+1},
\ene
where $\bn$ is from the right hand side of \eqref{eq_appropriate_bounds}. In the special case $\level(\bn_j)=1$, we have $|\CL_j|<C_{\safe}$, and the second inequality can be improved to
$$
\|\bn\cdot\omega\|\ge C_{\dio}|C_{\dist}C_{\safe}|^{-\tau}.
$$
Thus, one can achieve \eqref{eq_appropriate_bounds} by choosing a small $\beta$ (depending on $C_{\safe}$) and applying the first inequality in \eqref{eq_appropriate_2}.

Let us now consider the case $\level(\bn_j)\ge 2$. In this case, one can reduce \eqref{eq_appropriate_bounds} to
$$
\beta^{\level(\bn_j)-1}\le \frac{1}{16} \beta^{\level(\CL_j)}.
$$
Recall that $\CP$ is a loop stack, and therefore $\CL_j$ is a short loop: $|\CL_j|<\safedist(\bn_j)$. The bound \eqref{eq_l_bound} implies
\bee
\label{eq_h2_bound}
\level(\CL_j)\le \maxlevel(|\CL_j|)\le \maxlevel(\safedist(\bn_j))\ll \level(\bn_j),
\ene
assuming again that $\beta_{\max}$ is small enough (depending on $C_{\safe}$ and other parameters specified in the statement). Thus, in both cases we have \eqref{eq_appropriate_bounds}. If we are applying the induction assumption (2) to each $\CP_j$, we use the bound
\bee
\label{eq_attachment_cancellation}
|(V_{\bze}-V_{\bn_j}(t))^{-1}|\le \frac{1}{D_{\min}\|\bn_j\cdot\omega\|}.
\ene
The difference $|\cont([\CP_j],t+\bn_j\cdot\omega)-\cont([\CP_j],t)|$ can be estimated by the bound (2) on the derivative, times $\|\bn_j\cdot\omega\|$. As a consequence, we have
\bee
\label{eq_differentiation_factor}
\left|\frac{\cont([\CP_j],t+\bn_j\cdot\omega)-\cont([\CP_j],t)}{V_{\bze}-V_{\bn_j}(t)}\right|\le C_{\reg}\max\left\{\frac{4}{D_{\min}\beta^{\level(\CL_j)+1}},1\right\}\cdot (1)_j,
\ene
where $(1)_j$ denotes the right hand side of (1) for the induction assumption for $\CP_j$, and $\CL_j$ is the base loop of $\CP_j$; we have also cancelled $\|\bn_j\cdot\omega\|$. We see that we obtain $C_{\reg}$ each time the differentiation happens, and we do not gain $D_{\min}$ in the numerator. The latter is important since, in the high energy region, $D_{\min}$ can be very large.

To complete the proof of (1), note that the factor \eqref{eq_differentiation_factor} appears $\downedges(\CP,\mbe)$ times. Out of these, the first argument in the $\max$ function is chosen $\nbloops(\CP,\mbe)$ times, and the second argument appears $\downedges(\CP,\mbe)-\nbloops(\CP,\mbe)$ times.

To show (2), consider the derivative of \eqref{eq_prod}, and suppose that the product runs over $j=1,\ldots,m$. As a result of differentiating the product, we obtain $2m+1$ terms. The estimate from differentiating $\cont(\CL,t)$ or $(V_{\bze}-V_{\bn_j}(t))^{-1}$ follows from Lemma \ref{lemma_safe_loop_contrib} or Lemma \ref{lemma_single_denominator}, plus the induction assumption (1) for the remaining factors combined with \eqref{eq_simple_induction}. 

In the case of differentiating the last factor $\cont([\CP_j],t+\bn_j\cdot\omega)-\cont([\CP_j],t)$ in \eqref{eq_prod}, we estimate the derivatives of both terms by absolute value:
$$
\left|\frac{d}{dt}\cont([\CP_j],t+\bn_j\cdot\omega)-\frac{d}{dt}\cont([\CP_j],t)\right|\le \left|\frac{d}{dt}\cont([\CP_j],t+\bn_j\cdot\omega)\right|+\left|\frac{d}{dt}\cont([\CP_j],t)\right|.
$$
Each time we do the last operation, we gain a factor of $2$. Ultimately, the power of 2 is equal to the number of attachment points and can therefore be absorbed into $16^k$. We also get a factor $2m+1<5k$ each time we apply the induction. The total contribution from the last factors is bounded by $(5k)^{\height(k)}$, which can ultimately also be absorbed into $16^k$ (recall that $\height(k)$ is an extremely slowly growing function).

\end{proof}
\subsection{The main result}
Suppose that $\CP$ is a loop stack, $|\CP|=k$. Then we have the following:
$$
\den(\CP,\mbe)\le \frac{k}{\max\{1,C_{\dist}C_{\safe}\mbe^2\}};
$$
$$
\nbloops(\CP,\mbe)\le \loops(\CP,\mbe)\le \frac{k}{2\exp\{c(\dist_{\varphi},C_{\dio},\tau)\max\{\mbe,0\}\}};
$$
$$
\nblevel(\CP,\mbe)\le \totallevel(\CP,\mbe)\le \frac{C_{\dist}k}{C_{\safe}\max\{\mbe,1\}}.
$$
$$
\downedges(\CP,\mbe)\le k/2.
$$
Here, as above, $C_{\dist}$ denotes some constant that only depends on the distance function and can be different from one estimate to another. Note that the bound on $\downedges(\CP,\mbe)$ cannot be improved much, since once the base loop of $\CP$ reaches a small denominator of level $M$, we can attach an unlimited amount of shortest possible loops (that is, of length $2$) to that denominator. Recall now the main bound from Theorem \ref{th_loop_stack}
$$|\cont([\CP],t)|\le (C_{\dist}\|\varphi\|_{\infty})^k 
	\left(\frac{4}{D_{\min}}\right)^{\denominators(\CP,\mbe)} 
\beta^{-\totallevel(\CP,\mbe)-\denominators(\CP,\mbe)}\times
$$
$$
\times C_{\reg}^{\downedges(\CP,\mbe)}\left(\frac{4}{D_{\min}}\right)^{\nbloops(\CP,\mbe)}\beta^{-\nblevel(\CP,\mbe)-\nbloops(\CP,\mbe)}.
$$
Combining it with the previous estimates, we obtain (assuming additionally $D_{\min}\ge 1$)
$$
|\cont([\CP],t)|\le C_{\reg}^{k/2}(C_{\dist}\|\varphi\|_{\infty})^k\beta^{-\frac{3k}{2}\left(1+\frac{C_{\dist}}{C_{\safe}}\right)}.
$$
Recall that these bounds rely on $\beta$ being sufficiently small in order to satisfy the consistency condition $\eqref{eq_epsilon_safe}$:
$$
0<\beta<\beta_{\max}(C_{\mathrm{safe}},C_{\mathrm{dio}},\dist_{\varphi},\tau)
$$
and a bound of the form
$$
C_{\mathrm{safe}}\maxlevel(k)^3\ll k,
$$
which is required for \eqref{eq_h_bound} and \eqref{eq_h2_bound} and is possible due to \eqref{eq_l_bound} (see also Remark \ref{rem_multilog}). If $M_{\beta}>0$, then all the bounds only become better. Note that $D_{\min}$ can be arbitrarily large, and it is important that it is always in the denominator. We arrive at the following main result.
\begin{thm}
\label{main}
Suppose, $f$ is $C_{\mathrm{reg}}$-regular at $x_0\in (-1/2,1/2)\setminus(\Z+1/2+\Z^d\cdot\omega)$ and $D_{\min}\ge 1$. Consider the operator \eqref{eq_h_def}:
$$
(H(x_0)\psi)_{\bn}=\varepsilon\sum_{\bm\in \Z^d}\varphi_{\bn-\bm}\psi_{\bm}+f(x_0+\bn\cdot\omega)\psi_{\bn}.
$$
Let $V_{\bn}=f(x_0+\bn\cdot\omega)$, and consider the perturbation series
\bee
\label{eq_lak_mainth}
\lambda_0+\varepsilon\lambda_1+\varepsilon^2\lambda_2+\ldots,
\ene
$$
\psi=\psi_0+\varepsilon \psi_1+\varepsilon^2\psi_2+\ldots
$$
with $\lambda_0=f(x)$ and $\psi_0=e_{\bze}$, constructed in Theorem $\ref{th_lak}$. There exists $\gamma=\gamma(\tau,C_{\dio},\dist_{\varphi})$ such that the the terms of the eigenvalue series satisfy
\bee
\label{eq_lak_bound}
|\lambda_k|\le C_{\reg}^{k/2} (C_{\dist}\|\varphi\|_{\infty})^k\gamma^{-k}.
\ene
Suppose, in addition, that $\varepsilon<C_{\reg}^{-1/2}(C_{\dist}\|\varphi\|_{\infty})^{-1}\gamma$, so that the series converges. Then
\bee
\label{uniform_localization}
|\psi_{\bn}|=|\langle e_{\bn},\psi\rangle|\le \left(\frac{C_{\dist} \|\varphi\|_{\infty}\sqrt{C_{\reg}}}{\gamma}\right)^{\dist_{\varphi}(\bn,\bze)}\varepsilon^{\dist_{\varphi}(\bn,\bze)},
\ene
so that $\psi\in \ell^2(\Z^d)$ and it satisfies the eigenvalue equation
$$
H(x_0)\psi=\lambda\psi.
$$
\end{thm}
\begin{proof}
Recall that the coefficient at $\varepsilon^k$ at the eigenvalue series is a sum of the following terms:
$$
\cont([\CP])=\cont(\CP_0)\prod_{s=1}^m (V_{\bn_s}-V_{\bze})^{-1} \cont([\CP_s]),
$$
Each $\cont([\CP_s])$ factor can be estimated using Theorem \ref{th_loop_stack}. Due to the Diophantine condition, extra denominators $V_{\bn_s}$ can be absorbed into the bound.

For the eigenvector series, the same argument is sufficient to obtain the convergence of the series for each component of the eigenvector. However, note also that the most ``efficient'' way to contribute into $\psi_{\bn}$ is to consider a path that only has a safe base loop, since each attachment ``resets'' the route, and, due to cancellations, we do not need to worry about multiple attachments at small denominators. The bound \eqref{uniform_localization} follows from the safe loop bound similar to Lemma \ref{lemma_safe_loop_contrib}.
\end{proof}
\begin{cor}
\label{cor_lambda}
Suppose that $f$ is (uniformly) $C_{\reg}$-regular on an interval $(a,b)\subset (-1/2,1/2)$. Denote by $\lambda(x)$, where $x\in (a,b)\setminus(\Z+1/2+\Z^d\cdot\omega)$, the result of applying Theorem $\ref{main}$ to the operator $H(x)$. Then, the function $\lambda(x)$ extends to a continuous strictly monotone function on $(a,b)$.
\end{cor}
\begin{proof}
We have, from \eqref{eq_lak_mainth},
$$
\lambda(x)=f(x)+\varepsilon \lambda_1(x)+\varepsilon^2\lambda_2(x)+\ldots.
$$
The first term is strictly monotone in $x$ with a positive lower bound on the derivative. The derivatives of the remaining terms can be estimated in a manner similar to \eqref{eq_lak_bound}, using part (2) of Theorem \ref{th_loop_stack}. Therefore, after possibly choosing a smaller $\varepsilon$, $\lambda(x)$ will have the same lower bound on the derivative (which can be made arbitrarily close to that of $f$).
\end{proof}
\begin{cor}
\label{cor_psi}
Suppose, $f$ is $C_{\reg}$-regular on $(-1/2,1/2)$ and $D_{\min}\ge 1$. Let $x_0\in (-1/2,1/2)\setminus(\Z+1/2+\Z^d\cdot\omega)$ and $\psi[\bn]_{\bm}:=\psi(x_0+\bn\cdot\omega)_{\bm}$. Here, $\lambda(x)$ is defined in the same way as in Corollary $\ref{cor_lambda}$, and $\psi(x)$ as in Theorem $\ref{main}$. Then
\bee
\label{eq_eigenfunction_x0}
H(x_0)\psi[\bn]=\lambda_{\bn}\psi[\bn],
\ene
where $\lambda_{\bn}=\lambda(x_0+\bn\cdot\omega)$. The spectrum of $H(x_0)$ is pure point, and the above eigenvectors form a complete system.
\end{cor}
\begin{proof}
The eigenfunction equation \eqref{eq_eigenfunction_x0} follows from a standard translational covariance computation. Consider the operator $U$ with columns $\psi[\bn]$. Clearly,
$$
\|U U^*-I\|\le c\varepsilon,\quad \|U^*U-I\|\le c\varepsilon,
$$
where $c=c(C_{\safe},C_{\dio},\dist_{\varphi},\|\varphi\|_{\infty},\tau)$. Therefore, the span of the eigenvectors $\psi[\bn]$ is dense in $\ell^2(\Z^d)$. Corollary \ref{cor_lambda} implies that the spectrum is simple.
\end{proof}
\begin{rem}
The inverse function $\lambda^{-1}(E)$ is equal to $N(E)-1/2$, where $N(E)$ is the integrated density of states of the ergodic operator family $H(x)$. The easiest way to see that is to define IDS as the expectation value of the spectral measure.
\end{rem}
\begin{rem}
One can easily check that if $f$ is a meromorphic function in a neighborhood of $\R$, then both $\lambda(x)$ and $N(E)$ are also meromorphic. 
\end{rem}
\begin{rem}
\label{rem_uniform_localization}
Under the assumptions of Corollary \ref{cor_psi}, one can normalize the eigenfunctions $\psi[\bn]$ in $\ell^{\infty}(\Z^d)$ in order to satisfy {\it uniform localization}:
\bee
\label{eq_uniform_localization}
\psi[\bn]_{\bn}=1;\quad |\psi[\bn]_{\bm}|\le c \varepsilon^{\left\lfloor\frac{|\bn-\bm|}{c_{\dist}} \right\rfloor},\quad \bm\neq \bn,
\ene
where $c$ and $c_{\dist}$ can be picked uniformly in the region of applicability of Theorem \ref{main}, that is, whenever $f$ is $C_{\reg}$-regular at $\{x_0+\bn\cdot\omega\}$.
In other words, each eigenvector is a small perturbation of $e_{\bn}$, and its components decay exponentially in the distance from the ``localization center'' $\bn$, with an upper bound that is uniform for all eigenvectors.
\end{rem}
\begin{rem}
The fact that $\lambda(x)$ extends continuously into $(-1/2,1/2)$ has the following meaning: suppose that $x\in \Z+\bn\cdot\omega+1/2$. Then the potential becomes infinite at $\bn$. As usual, one can understand it as a Dirichlet boundary condition imposed at $\bn$. For the usual Schr\"odinger operator on $\Z$ with nearest neighbor hopping, the problem thus splits into a direct sum of two half-line operators, each having pure point spectrum.
\end{rem}
\subsection{The local result} Theorem \ref{th_loop_stack} only relies on $C_{\reg}$-regularity of $f$ at one point $x$. In particular, it does not require monotonicity of $f$ outside the interval $(a,b)$ (although it still needs the ``regularity at infinity'' condition (cr2)). As a consequence, at any such point $x$ the perturbation series for the eigenvalue and the eigenfunction will converge. If $f$ is $C_{\reg}$-regular on an interval, it implies convergence of the perturbation series for all eigenvalues and eigenfunctions whose (unperturbed) energies fall into the image of that interval. The following results state that, if one slightly decreases the size of the interval, then these eigenfunctions will actually exhaust the spectral projection of $H$ in the smaller interval. In other words, regularity of $f$ on an interval implies Anderson localization for $H$ on the same interval.

\begin{lem}
\label{lemma_local_localization}
Suppose that $f$ is $C_{\reg}$-regular on $(\alpha,\beta)$, where $-1/2<\alpha<\beta<1/2$ and $D_{\min}\ge 1$. Fix $x_0\in (-1/2,1/2)\setminus(\Z+1/2+\Z^d\cdot\omega)$ and apply Theorem $\ref{main}$ for all $x=x_0+\bn\cdot\omega\in (\alpha,\beta)$. In the notation of Corollary $\ref{cor_psi}$, denote
$$
\Pi_{(\alpha,\beta)}=\overline{\mathrm{span}\{\psi[\bn]\colon \{x_0+\bn\cdot\omega\}\in (\alpha,\beta),\bn\in \Z^d\}}.
$$
There exist $c=c(C_{\reg},C_{\dio},\dist_{\varphi})>0$ and $\varepsilon_0=\varepsilon_0(C_{\reg},C_{\dio},\dist_{\varphi},\|\varphi\|_{\infty})$ such that, for $0<\varepsilon<\varepsilon_0$ and all $\psi\perp \Pi_{(\alpha,\beta)}$, we have
$$
\left\|\left(H(x_0)-\frac{f(\alpha)+f(\beta)}{2} I\right)\psi\right\|\ge \frac{f(\beta)-f(\alpha)}{2} (1-c\varepsilon)\|\psi\|.
$$
\end{lem}
\begin{proof}
Denote by $E$ the orthogonal projection onto the subspace
$$
\mathrm{Ran}\, E =\overline{\mathrm{span}\{e_{\bn}\colon \{x_0+\bn\cdot\omega\}\in (\alpha,\beta),\bn\in \Z^d\}}.
$$
Let also $U\colon \ell^2(\Z^d)\to \ell^2(\Z^d)$ be defined on the standard basis by
$$
Ue_{\bn}=\begin{cases}
\psi[\bn],& \{x_0+\bn\cdot\omega\}\in (\alpha,\beta)\\
0,& \{x_0+\bn\cdot\omega\}\notin (\alpha,\beta).
\end{cases}
$$
One can check, using \eqref{eq_uniform_localization}, that
\bee
\label{eq_uustar}
\|U U^*-E\|\le c_1\varepsilon,
\ene
where $c_1$ depends on $c$ and $c_{\dist}$ from \eqref{eq_uniform_localization}. Let $\psi\perp \Pi_{(\alpha,\beta)}$. Then $U^*\psi=0$, and \eqref{eq_uustar} implies $\|E\psi\|\le c_1\varepsilon\|\psi\|$. Recall that $H(x_0)=V(x_0)+\varepsilon\Phi$, where $V(x_0)$ is an operator of multiplication by $f(x_0+\bn\cdot\omega)$, and let $\psi'=E\psi$, $\psi''=(1-E)\psi$. From the definition of $E$ it follows that
$$
\left\|\left(V(x_0)-\frac{f(\beta)+f(\alpha)}{2}\right) \psi''\right\|\ge \frac{f(\beta)-f(\alpha)}{2}\|\psi''\|\ge \frac{f(\beta)-f(\alpha)}{2}(1-c_1 \varepsilon)\|\psi\|,
$$
$$
\left\|\left(V(x_0)-\frac{f(\beta)+f(\alpha)}{2}\right) \psi'\right\|\le \frac{f(\beta)-f(\alpha)}{2}\|\psi'\|\le \frac{f(\beta)-f(\alpha)}{2}c_1 \varepsilon\|\psi\|.
$$
By combining the previous bounds, we obtain
\begin{multline*}
\left\|H(x)\psi-\frac{f(\alpha)+f(\beta)}{2}\psi\right\|\\
\ge \left\|\left(V(x)-\frac{f(\beta)+f(\alpha)}{2}\right) \psi''\right\|-\varepsilon\|\Phi \psi\|-\left\|\left(V(x)-\frac{f(\beta)+f(\alpha)}{2}\right) \psi'\right\| \\
\ge \left(\frac{f(\beta)-f(\alpha)}{2}(1-c_1\varepsilon)-\varepsilon\|\Phi\|-\frac{f(\beta)-f(\alpha)}{2}c_1 \varepsilon\right)\|\psi\|,
\end{multline*}
from which the claim follows.
\end{proof}
The original definition of $C_{\reg}$-regularity may be too restrictive if the interval $(f(\alpha),f(\beta))$ is small, since it requires monotonicity of $f$ in a large neighborhood of that interval. It is more convenient to formulate a general result using the rescaled version of regularity in the sense of \ref{rem_creg}:
\begin{thm}
\label{th_local_localization}
Suppose that $f$ is $C_{\reg,\nu}$-regular on $(\alpha,\beta)$, where $-1/2<\alpha<\beta<1/2$ and $D_{\min}(x)\ge D_0>0$ on $(\alpha,\beta)$. There exists $\ep_0=\ep_0(C_{\reg},C_{\dio},\dist_{\varphi},\|\varphi\|_{\infty},\nu,D_0)$ such that, for $0<\ep<\ep_0$ and all $x_0\in (-1/2,1/2)\setminus(\Z+\Z^d\cdot\omega)$, the spectrum of $H(x_0)$ in the interval $[f(\alpha),f(\beta)]$ is pure point, and the eigenfunctions $\psi[\bn]$, whose energies are in that interval, form a complete system in the range of the spectral projection onto $[f(\alpha),f(\beta)]$.
\end{thm}
\begin{proof}
Since $f$ is $C_{\reg,\nu}$-regular on $(\alpha,\beta)$, we can assume that $f$ is $C_{\reg,\nu/2}$-regular in a slightly larger interval $(\alpha',\beta')$. Now the result follows from applying Lemma \ref{lemma_local_localization} to the rescaled function $(\nu/2)^{-1}f$ and choosing an appropriately small $\ep$. The case of arbitrary $D_0$ can also be treated by a similar rescaling.
\end{proof}
\section{Generalizations: the long-range case and non-constant hopping terms}
\subsection{Non-constant long range hopping terms}A {\it quasiperiodic hopping matrix} is, by definition, a matrix with elements of the following form
\bee
\label{eq_quasi_1}
\Phi_{\bm\bn}(x)=\varphi_{\bm-\bn}(x+(\bm+\bn)\cdot\omega/2),\quad \bm,\bn\in \Z^d,
\ene
where $\varphi_{\bm}\colon \mathbb R\to \mathbb C$ are Lipschitz $1$-periodic functions, satisfying the self-adjointness condition:
$$
\varphi_{\bm}=\overline{\varphi_{-\bm}}.
$$
Let also
$$
\|\varphi\|_{\infty}=\sup_{\bk} \|\varphi_{\bk}\|_{\infty},\quad\|\varphi'\|_{\infty}=\sup_{\bk} \|\varphi'_{\bk}\|_{\infty}.
$$
Define $\range(\Phi)$ to be the smallest number $L\ge 0$ such that $\Phi_{\bn\bm}\equiv 0$ for $|\bm-\bn|>L$. We will only consider hopping matrices of finite range. Note that \eqref{eq_quasi_1} can be reformulated as the following covariance property:
$$
\Phi_{\bm+\ba,\bn+\ba}(x)=\Phi_{\bm\bn}(x+\ba\cdot\omega),\quad \bm,\bn,\ba\in \Z^d.
$$
Fix some $R\in \mathbb N$, and suppose that $\Phi^1,\Phi^2,\ldots$ is a family of quasiperiodic hopping matrices with $\range(\Phi_k)\le k R$, defined by a family of functions $\varphi^1_{\bm}, \varphi^2_{\bm}, \ldots$. A more general class of operators we would like to consider will be of the following form:
\bee
\label{eq_longrange_def}
H=V+\varepsilon\Phi^1+\varepsilon^2\Phi^2+\ldots,\quad 0\le \varepsilon<1,
\ene
where, as previously,
$$
(V\psi)_{\bn}=V_{\bn}\psi_{\bn}=f(x_0+\bn\cdot\omega)\psi_{\bn}.
$$
One can easily check that, assuming 
$$
\|\varphi\|_{\infty}=\sup_{j}\|\varphi^j\|_{\infty}=\sup_{j,\bm}\|\varphi^j_{\bm}\|_{\infty}<+\infty,\quad 0\le\varepsilon<1,
$$ 
the part $\varepsilon\Phi^1+\varepsilon^2\Phi^2+\ldots$ defines a bounded operator on $\ell^2(\Z^d)$. Similarly to the previous sections, we are dealing with the equations
$$
(V+\ep \Phi^1+\ep^2\Phi^2+\ldots)(\psi_0+\ep\psi_1+...)=(\la_0+\ep\la_1+...)(\psi_0+\ep\psi_1+...),
$$
and the initial conditions
$$
\lambda_0=V_{\bze}=f(x_0),\quad \psi_0=e_{\bze},\quad \psi_j\perp\psi_0\,\,\text{for}\,\,j>0.
$$
Formal equalizing of terms at $\varepsilon^k$ leads to
$$
V \psi_k+\Phi^1\psi_{k-1}+\ldots+\Phi^{k-1} \psi_1+\Phi^k e_{\bze}=\lambda_0\psi_k+\lambda_1\psi_{k-1}+\lambda_2\psi_{k-2}+\ldots+\lambda_{k-1}\psi_1+\lambda_k e_{\bze}
$$
which, after projecting onto $\mathrm{span}\{e_{\bze}\}$ and $\mathrm{span}\{e_{\bze}\}^{\perp}$, is reduced to
\bee
\label{eq_psik_infrange}
\psi_k=(V_{\bze}-V)^{-1}\left(\Phi^1\psi_{k-1}+\ldots+\Phi^{k-1} \psi_1+\Phi^k e_{\bze}-\lambda_1 \psi_{k-1}-
\lambda_2\psi_{k-2}-\ldots-\lambda_{k-1}\psi_1\right);
\ene
\bee
\label{eq_lak_infrange}
\lambda_{k}=\langle \Phi^1 \psi_{k-1}+\Phi^2\psi_{k-2}+\ldots+\Phi^{k-1} \psi_1,e_{\bze}\rangle.
\ene
In order to establish an analogue of Theorem \ref{th_lak}, one needs to construct an appropriate version of the graph $\Gamma$. It will formalize the following construction: each time a path makes a jump between two vertices, it can ``choose'' which term of the kinetic energy will be used to perform that jump. Naturally, the use of the term $\Phi^j$ will cost $\ep^j$; in other words, the corresponding graph edge will contribute $j$ to the length of the path. In the construction from Section 3, all edges had length $1$.
\begin{itemize}
	\item Take $\Phi=\Phi^j$ and denote the graph constructed in Section 3.1 using only that $\Phi$, by $\Gamma_j=(\mathcal E,\mathcal V_j)$; here $\mathcal E$ is the set of vertices (same for each $\Gamma_j$), and $\mathcal V_j$ is the set of edges. Let
$$
\Gamma:=(\mathcal E,\bigsqcup_j \mathcal V_j),
$$
where $\bigsqcup$ denotes the disjoint union. In other words, $\Gamma$ is a graph on the same set of vertices as each of the $\Gamma_j$, and includes all edges of each $\Gamma_j$, considered separately (thus obtaining a multi-graph). To each edge of $\Gamma$, we associate two numbers: the length (an integer number $j$ for an edge inherited from $\Gamma^j$) and the weight (defined in Section 3.1 with $\Phi=\Phi^j$). We will also be following the convention from Remark \ref{rem_zero_weights} for future convenience.

	\item As a result, a path in the graph $\Gamma$ is no longer determined by a sequence of vertices, and requires additional data: for each jump (by a jump we will mean moving from one vertex in our sequence to the next one), one needs to specify the length of the edge used for that jump. Any edge of $\Gamma$ is uniquely determined by its length and starting and ending vertices.
	\item The loop translation is now defined in a way that, whenever a piece of a path gets translated, every edge gets translated into an edge of the same length as it was before translation. The convention in Remark \ref{rem_zero_weights} guarantees that the translated path will exist. However, since $\Phi^j_{\bm\bn}$ depends on $x$ and is no longer translation invariant, a path with a non-zero contribution may be translated into a path with zero contribution.
\end{itemize}
If $\CP$ is an eigenvalue/eigenvector path on $\Gamma$, we will denote by $|\CP|$ the total length of edges of $\CP$. With the above conventions in place, one can reformulate the main results of Section 2:
$$
\la_s=\sum_{\CP,|\CP|=s} \cont(\CP),
$$
where the sum is considered over all eigenvalue loop configurations $\CP$ (with $|\CP|=s$), and a similar expression for the eigenvector
$$
(\psi_s)_{\bk}=\sum_{\CP,|\CP|=s} \cont(\CP),\,\, \bk\neq 0,\,\, s>0;
$$
The definition of a safe loop follows Definition \ref{def_safe_loop}, with the addition that the number of jumps is replaced by the total length of the edges. In the definition of $\cont(\CP,t)$, in addition to replacing $V_{\bn}=f(x_0+\bn\cdot\omega)$ by $V_{\bn}(t)=f(x_0+t+\bn\cdot\omega)$ for $\bn\neq \bze$, we also replace $\Phi^j_{\bm\bn}(x_0)$ by 
$$
\Phi^j_{\bm\bn}(x_0,t):=\Phi^j_{\bm\bn}(x_0+t)=\varphi_{\bm-\bn}(x_0+t+(\bm+\bn)\omega/2).
$$

%We will say that $f$ and $\{\varphi_{\bn}:\bn\in\mathbb Z^d\}$ are {\it uniformly $C_{\mathrm{reg}}$-regular} if, in the notation of Section 3,
%\begin{itemize}
%\item[(ucr0)]The pre-image $f^{-1}((f(x_0)-2,f(x_0)+2))\cap (-1/2,1/2)$ is an open interval (denoted by $(a,b)$), and $\left.f\right|_{(a,b)}$ is a one-to-one map between $(a,b)$ and $(f(x_0)-2,f(x_0)+2)$.
%\item[(ucr1)]Let $D_{\min}(x_0)=\inf\limits_{x\in (a,b)}f'(x)$ (for points where $f'$ does not exist, consider the smallest of the derivative numbers). Then, assume 
%\bee
%\label{eq_ucr1}
%D_{\min}(x_0)\le f'(x)\le C_{\mathrm{reg}} D_{\min}(x_0),\quad \forall x\in (a,b).
%\ene
%\item[(ucr2)]Define $(a_1,b_1)=f^{-1}((f(x_0)-1,f(x_0)+1))\subset (a,b)$, and
%$$
%g_{y,\bm}(x)=\frac{\varphi_{\bm}(x+y)}{f(x)-f(x_0)}, \quad x\in (b_1,a_1+1),
%$$
%extended by continuity to $g(\pm 1/2)=0$. Similarly to (cr2), we require
%$$
%|g_{y,\bm}'(x)|\le C_{\mathrm{reg}} D_{\min},\quad x\in (b_1,a_1+1).
%$$
%uniformly in $y$ and $\bm\in\Z^d$.
%\end{itemize}

The results of Section 4 can be repeated without significant changes. It is helpful to observe the following: one can treat a jump from $\bm$ to $\bn$, say, with $\dist_{\Phi}(\bm,\bn)=3$, as three jumps of distance $1$, where only one of these jumps makes an actual contribution to $\cont(\CP)$, and the contributions of two other jumps are equal to $1$. The resulting bounds are only better (since there are less denominators to care about). Since $\varphi$ is no longer a constant, each time we are estimating a derivative we have a choice of differentiating one of the $\varphi$ factors. As a consequence, the factor $\|\varphi\|_{\infty}$ in front has to be replaced by $\|\varphi\|_{\infty}+\|\varphi'\|_{\infty}$. An extra combinatorial factor can be absorbed into $C_{\dist}^{k}$.
\begin{prop}
\label{prop_nonlocal_step}
Let $\CP$ be an eigenvalue/eigenvector loop stack with $k$ edges. Suppose, $f$ is $C_{\mathrm{reg}}$-regular at $x_0$ and $\CL$ is the base loop of $\CP$. 
Denote $\mbe$ by $\eqref{eq_m_definition}$, and let
$$
|t|\le  \frac14 \min_{\bn\in \CL}\|\bn\cdot\omega\|,\quad 0<\beta< \beta_{\max}(\dist_{\varphi},C_{\dio},\tau_{\dio},C_{\safe}).
$$
Then
\begin{enumerate}
	\item $$|\cont([\CP],t)|\le C_{R}^{|\CP|}(\|\varphi\|_{\infty}+\|\varphi'\|_{\infty})^k
	\left(\frac{4\gamma}{D_{\min}}\right)^{\denominators(\CP,\mbe)} 
\beta^{-\totallevel(\CP,\mbe)-\denominators(\CP,\mbe)}\times
$$
$$
\times C_{\reg}^{\downedges(\CP,\mbe)}\left(\frac{4}{D_{\min}}\right)^{\nbloops(\CP,\mbe)}
\beta^{-\nblevel(\CP,\mbe)- \nbloops(\CP,\mbe)}.
$$
\item The function $\cont([\CP],\cdot)$ is Lipschitz continuous with the bound on the derivative
$$\left|\frac{d}{dt}\cont([\CP],t)\right|\le C_{\reg} \max\left\{\frac{4\gamma}{\varepsilon^{\level(\CL)+\delta_{\mathrm{lev}}}},D_{\min}\right\}\cdot(\text{r.h.s of }(1)).
$$
\end{enumerate}
\end{prop}

\begin{rem}
The combinatorial factor now depends only on $R$, since the distance function and the range of $\Phi^j$ is defined by $R$. Note that, despite multiple choice of edges, the total number of paths of length $|\CP|$ with non-zero contributions is still bounded by $C_R^{|\CP|}$. The power of $\|\varphi\|_{\infty}$ now depends only on $k$, which is the number of edges. The fact that one also needs to differentiate $\varphi$ only affects the result by adding $\|\varphi'\|_{\infty}$ into the factor $(\|\varphi\|_{\infty}+\|\varphi'\|_{\infty})$, since differentiating $\varphi$ means that we are using $\|\varphi'\|_{\infty}\cdot(\text{r.h.s of }(1))$ instead of $(\text{r.h.s of }(2))$. These possible extra $\|\varphi'\|_{\infty}$ can be absorbed into the factors in front of (1).
\end{rem}
\begin{rem}
The reader can easily formulate appropriate versions of the results of Section 4, starting from Theorem \ref{main}. The only differences are the inclusion of $\|\varphi'\|_{\infty}$ and different notation for the distance function.
\end{rem}
\subsection{The case of monotone potentials with small derivatives}
The results of Section 6.1 can easily be extended to the case where the terms $\Phi^j$ of \eqref{eq_longrange_def} have some additional dependence on $\ep$, as long as the estimates of $\|\varphi\|_\infty$ and $\|\varphi'\|_{\infty}$ are uniform in $\ep$ in the range of parameters under consideration. The same arguments can be applied to the dependence of $f$ on $\ep$ if one considers both upper and lower bounds on the derivatives described in the definitions of regularity in Section 5.1. In particular, we always assume that $D_{\min}\ge 1$ (although rescaling allows to replace $1$ by any positive $\ep$-independent quantity). In some of the applications, these assumptions are not flexible enough. In particular, the operators considered in Section 7, for which $f$ is originally not strictly monotone, can be transformed into those with monotone $f$, but the derivative will only admit a lower bound by $c\varepsilon^2$ on some intervals. The bounds outlined above, essentially, contain a factor $D_{\min}^{-k}$, leading to a coefficient of the order $\ep^{-2k}$ at $\ep^k$, not sufficient to guarantee convergence of the series for such $f$. In addition to that, $C_{\reg}$ will also contain negative powers of $\ep$.

In order to deal with these problems, we introduce some refinements to our general scheme which we will now describe. They will involve additional restrictions on the hopping terms and $f$. Assume the following:

\begin{itemize}
	\item[(sing0)]$I_1,\ldots,I_k\subset (-1/2,1/2)$ are open intervals whose closures are mutually disjoint.
	\item[(sing1)] $f$ satisfies (f1) from Section 3. Additionally, $f$ is $C_{\reg}$-regular on $(-1/2,1/2)\setminus (I_1\cup\ldots\cup I_k)$ with $D_{\min}(x_0)\ge 1$ for all $x_0\in (-1/2,1/2)\setminus (I_1\cup\ldots\cup I_k)$. For each $I_j=(a_j,b_j)$, let $D_{j,+}:=\max\{D_{\min}(a),D_{\min}(b)\}$, $D_{j,-}:=\min\{D_{\min}(a),D_{\min}(b)\}$.
	\item[(sing2)] For some $\delta>0$, we have $c_{j,-} \delta^{\mu_j}\le f'(x)\le c_{j,+}$ on $I_j$. As usual, we assume that the inequalities hold for all derivative numbers in the case $f$ is not differentiable. As a consequence, $f$ is strictly monotone on $(-1/2,1/2)$.
	\item[(sing3)] For each $x_0\in I_j$, we have (cr2) with $D_{\min}(x_0)$ replaced by $c_{j,+}$.
\end{itemize}
Since $f$ is regular on $(-1/2,1/2)\setminus (I_1\cup\ldots\cup I_k)$, the results of the previous subsection apply for $x_0\in (-1/2,1/2)\setminus (I_1\cup\ldots\cup I_k)$ outside of these intervals. 

If $x_0\in I_j$, we can estimate the denominators by considering three cases:
\begin{enumerate}
	\item {\it Singular} small denominators: $t+x_0\in I_j$, $t+x_0+\bn\cdot\omega\in I_j$ for all $t$ under consideration. In this case, we have to use $c_{j,-}\delta^{\mu_j}$ as the lower bound for the derivative of $f$, and can use $c_{j,+}$ as an upper bound.
	\item {\it Regular} small denominators: $x_0\in I_j$, but either $t+x_0$ or $t+x_0+\bn\cdot\omega$ is not in $I_j$, for some $t$ under consideration. In this case, while $f$ is not regular at $x_0$, we can treat $x_0$ as a small perturbation of $t+x_0$ or $t+x_0+\bn\cdot\omega$. These denominators can be considered in the same way as the regular case, under some modifications of the constants: since some of the points are outside of $I_j$, one cannot use $c_{j,+}$ as an upper bound on the derivative but can use $C_{\reg}D_{j,+}$ as an upper bound and $D_{j,-}$ as a lower bound.
	\item For non-small denominators, in view of (sing3), one can use $C_{\reg}c_{j,+}$ as an upper bound on the derivative; lower bounds are not needed since the absolute value of the denominator is at least $1$.
\end{enumerate}
For simplicity, let us also assume that $\beta$ is chosen to be small enough so that all denominators of level $1$ or higher are actually small. In other words, $M_{\beta}(x_0)=0$ for all $x_0\in I_j$. We will drop $M_{\beta}$ from the notation: for example, $\den(\CL)$ will denote $\den(\CL,0)$. For a loop $\CL$, denote by $\singden(\CL)$ the total number of singular small denominators visited by $\CL$.
\begin{lem}
\label{lemma_safe_loop_contrib_smallderivative}
Under the above assumptions, let $\CL$ be an eigenvalue/eigenvector loop with $k$ edges. Let
$$
x_0\in I_j,\quad |t|\le \frac14 \min_{\bn\in \CL}\|\bn\cdot\omega\|.
$$
%Denote
%$$
%M=\frac{\log(4\gamma)-\log D_{\min}}{\log\varepsilon}-\dlev.
%$$
Then
\begin{enumerate}
	\item $|\cont(\CL,t)|\le \|\varphi\|_{\infty}^{k} \left(\frac{4}{D_{j,-}}\right)^{\denominators(\CL)} \beta^{-\totallevel(\CL)-\denominators(\CL)}\left(\frac{D_{j,-}}{\delta^{\mu_j} c_{j,-}}\right)^{\singden(\CL)}$.
	\item The function $\cont(\CL,\cdot)$ is Lipschitz continuous:
$$
\left|\frac{d}{dt}\cont({\CL,t})\right|\le|\CL|\cdot \max\left\{\frac{4\delta^{-\mu_j}c_{j,+}}{c_{j,-}\beta^{\level(\CL)+1}},\frac{4C_{\reg}D_{j,+}}{D_{j,-}\beta^{\level(\CL)+1}},C_{\reg}c_{j,+},\right\}\cdot(\text{r.h.s of }(1)).
$$
\end{enumerate}
\end{lem}
\begin{proof}As described above, we consider three types of denominators. In the case of estimating (1), non-small denominators do not give any contribution (we use a lower bound by $1$ for each of them). The ``regular'' small denominators are treated in the same way as in the regular case, with the same contributions. In the case of singular small denominators, we replace $D_{j,-}$ by $(\delta^{c_{j,-}})^{\mu_j}$ which produces the last factor in (1) (``replacement factor'').

In the case of (2), we argue similarly to Lemma \ref{lemma_safe_loop_contrib}; that is, a differentiation would add an extra factor $|\CL|\left|\frac{V_{\bn}'(t)}{V_{\bze}-V_{\bn}(t)}\right|$ in the upper bound, and we estimate the ratio in the worst possible case, again, by considering the three types of denominators.
\end{proof}
In order to make further estimates less cumbersome, we assume that $\beta$ and $\delta$ are small enough, depending on $C_{\reg},D_{j,\pm},c_{j,\pm}$, so that in the conclusion (2) of Lemma \ref{lemma_safe_loop_contrib_smallderivative} we have
\bee
\label{eq_assumptions}
\frac{4\delta^{-\mu_j}c_{j,+}}{c_{j,-}\beta^{\level(\CL)+1}}\ge \frac{4C_{\reg}D_{j,+}}{D_{j,-}\beta^{\level(\CL)+1}}\ge C_{\reg}c_{j,+}.
\ene
It will be convenient to introduce some additional notation.
\begin{itemize}
	\item Let $\singdownedges(\CP)$ denote the number of descending edges of $\CP$ which lead to a singular small denominator.
	\item Let also $\singden(\CP)$ be the total number of singular small denominators visited by $\CP$.
\end{itemize}
\begin{thm}
\label{thm_smallderivative_loop}
Under the above assumptions on $f$ and $\Phi$, let $\CP$ be an eigenvalue/eigenvector loop stack with $k$ edges. Suppose that $x_0\in I_j$ and $\CL$ is the base loop of $\CP$. Let
$$
|t|\le  \frac14 \min_{\bn\in \CL}\|\bn\cdot\omega\|,\quad 0<\beta< c(\dist_{\varphi},C_{\dio},\tau_{\dio},C_{\safe},c_{j,\pm},D_{j,\pm},C_{\reg})
$$
$$
0<\delta<\delta_0(\dist_{\varphi},C_{\dio},\tau_{\dio},C_{\safe},c_{j,\pm},D_{j,\pm},C_{\reg}).
$$ 
Then we have
 $$|\cont([\CP],t)|\le C_{R}^{|\CP|}(\|\varphi\|_{\infty}+\|\varphi'\|_{\infty})^k
	\left(\frac{4}{D_{j,-}}\right)^{\denominators(\CP)}
\beta^{-\totallevel(\CP)-\denominators(\CP)}\times \left(\frac{D_{j,-}}{\delta^{\mu_j} c_{j,-}}\right)^{\singden(\CP)}\times
$$
$$
\times \left(\frac{4 C_{\reg} D_{j,+}}{D_{j,-}^2}\right)^{\downedges(\CP)}
\beta^{-\nblevel(\CP)- \downedges(\CP)}\left(\frac{c_{j,+}D_{j,-}^2}{\delta^{2\mu_j} c_{j,-}^2 D_{j,+}}\right)^{\singdownedges(\CP)}.
$$
%\item The function $\cont([\CP],\cdot)$ is Lipschitz continuous with the bound on the derivative
%$$\left|\frac{d}{dx}\cont([\CP],x)\right|\le C_{\reg} \max\left\{\frac{4\gamma}{\varepsilon^{\level(\CL)+\delta_{\mathrm{lev}}}},D_{\min}\right\}\cdot(\text{r.h.s of }(1)).
%$$
\end{thm}
\begin{proof}
The proof essentially repeats the proof of Theorem \ref{th_loop_stack}, with modified constants. The first factor comes from considering the individual loops on $\CP$ and applying Lemma \ref{lemma_safe_loop_contrib_smallderivative} to each of them.

The remaining factor comes from considering the attachment factors and the differentiation factors. For example, in the factor $C_{\reg}^{\downedges(\CP,\mbe)}\left(\frac{4}{D_{\min}(x_0)}\right)^{\nbloops(\CP,\mbe)}$ from Theorem \ref{th_loop_stack}, we note that $\downedges(\CP)=\nbloops(\CP)$, replace $D_{\min}(x_0)$ by $D_{j,-}$ and $C_{\reg}$ by $\frac{C_{\reg}D_{j,+}}{D_{j,-}}$, thus obtaining the factor $\left(\frac{4 C_{\reg} D_{j,+}}{D_{j,-}^2}\right)^{\downedges(\CP)}$. Together with $\beta^{-\nblevel(\CP)- \nbloops(\CP)}$, this would complete the bounds in the case of no singular small denominators.

For each singular small denominators, instead of the above procedure, we need to replace $C_{\reg}$ by $\frac{c_{j,+}}{c_{j,-}\delta^{\mu_j}}$, and $D_{\min}$ by $c_{j,-}\delta^{\mu_j}$, which produces the last replacement factor.
\end{proof}
In order to formulate the main result of this section, let us state an additional condition on $f$. We would like to emphasize that $\delta$ is allowed to depend on $\ep$.
\begin{itemize}
	\item[(sing4)] There exists $\beta_{\max}=\beta_{\max}(C_{\reg},I_j,D_{j,\pm},c_{j,\pm},\omega)>0$ such that, for $0<\beta<\beta_{\max}$ there exist 
	$$
	\ep_0=\ep_0(\beta,C_{\reg},I_j,D_{j,\pm},c_{j,\pm},\omega)>0,\quad 0<r=r(\beta,C_{\reg},I_j,D_{j,\pm},c_{j,\pm},\omega)<1
	,$$ 
	$$C=C(\beta,C_{\reg},I_j,D_{j,\pm},c_{j,\pm},\omega)>0
	$$ such that, for any loop stack $\CP$ and for any $j$ we have
	$$
	\delta^{\mu_j\singden(\CP)+2\mu_j \singdownedges(\CP)}\le C^{|\CP|}\ep^{r|\CP|}.
	$$
\end{itemize}
We can now state the last result of this section. Its proof, after the above preparations, repeats the corresponding steps of the proof of Theorem \ref{main}.
\begin{thm}
\label{th_main_singular}
Suppose that $f$ satisfies $(\mathrm{sing}0)$--$(\mathrm{sing}3)$ with $C_{\reg}$, $D_{j,\pm}$, $c_{j,\pm}$, $I_j$ independent of $\ep$. Under the assumptions of $(\mathrm{sing}4)$, and further assuming $$0<\ep<\ep_1(\beta,C_{\reg},I_j,D_{j,\pm},c_{j,\pm},\omega,C,c,\ep_0,r),$$ the perturbation series for $H(x_0)$ converges for all $x_0\in \R\setminus(\Z+1/2+\Z^d\cdot\omega)$.
\end{thm}
%\begin{itemize}
%\item An {\it long range {$\bm\bm$}-loop} is a finite sequence of lattice vectors $\CL=(\bn_0,\bn_1,...,\bn_{l+1})$,
%where $l\ge 0$, $\bn_0=\bn_{l+1}=\bm$, all other $\bn_j$ cannot be equal to $\bm$ or $\bze$, and for each 
%$j=0,...,l$ we have $\Phi_{\bn_{j+1},\bn_j}\ne 0$.
%\item A {\it long range {$\bze\bk$}-loop} is a finite sequence of lattice vectors $\CL=(\bn_0,\bn_1,...,\bn_{l+1})$,
%where $l\ge 0$, $\bn_0=\bze$, $\bn_{l+1}=\bk$, all other $\bn_j$ cannot be equal to $\bze$, and for each 
%$j=0,...,l$ we have $\Phi_{\bn_{j+1},\bn_j}\ne 0$.
%\item A {\it long range loop configuration} is one of the following:
%\begin{enumerate} 
%	\item A long range $\bze\bk$-loop with $\bk\in \Z^d$.
%	\item An expression of the form $\CP_-,\bn,\CL,\bn,\CP_+$, where $\CL$ is a long range $\bze\bze$-loop, and 
%$$
%\CP=\CP_-,\bn,\CP_+
%$$ 
%is a long range loop configuration that contains a lattice point $\bn\neq \bze$.
%\end{enumerate}
%\item The contribution terms are defined as follows:
%$$
%C(\bn_0,\bn_1,\ldots,\bn_{l},\bn_{l+1})=\prod\limits_{j=0}^{l}\Phi_{\bn_{j+1},\bn_{j}}\prod\limits_{j=1}^{l}(-V_{\bn_j})^{-1},\quad \bn_0=\bn_{l+1}=0.
%$$
%$$
%\label{eq_cp_basem}
%C(\bn_0,\bn_1,\ldots,\bn_{l},\bn_{l+1})=\prod\limits_{j=0}^{l}\Phi_{\bn_{j+1},\bn_{j}}\prod\limits_{j=1}^{l+1}(-V_{\bn_j})^{-1}, \quad 0=\bn_0\neq \bn_{l+1}=\bk.
%$$
%\item For long range loop configurations $\CP$, define $|\CP|$ to be the sum of $\layer(\bm,\bn)$ for each factor $\Phi_{\bm\bn}$ appearing in the expression. 
%\end{itemize}
\section{An example: monotone potential with a flat segment}
\subsection{The class of operators}
Fix an interval $(a-h,a+h)\subset (-1/2,1/2)$, where
$$
-1/2<a-h<a<a+h<1/2.
$$
Consider a continuous non-decreasing function $f\colon (-1/2,1/2)\to \mathbb R$ and a frequency vector $\omega$ with the following properties:
\begin{itemize}
\item $f(-1/2+0)=-\infty$, $f(1/2-0)=+\infty$.
\item $f$ is $C_{\reg}$-regular uniformly for $x\in (-1/2,a-h)\cup(a+h,1/2)$ for some $C_{\reg}>0$.
\item $f'(x)\ge 1$ for all $x\in (-1/2,a-h)\cup(a+h,1/2)$ and $0\le f'(x)\le c_+$ for all $x\in [a-2h,a+2h]$ for some $c_+>0$.
\item $\omega\in (-1/2,1/2)^d$ is a Diophantine frequency vector. In addition, $\|\bn\cdot\omega\|\ge 6h$, for all $\bn$ with $1\le |\bn|_{\infty}\le c_1$. {\it Note that this imposes an additional requirement on the smallness of $h$}. Here, $c_1$ is a sufficiently large absolute constant ($c_1=6$ is sufficient for the main result of this section; however, we did not intend to optimize it and believe that a small modification would allow to use $c_1=3$).
\end{itemize}
\begin{rem}
\label{rem_outside_interval}
This class of functions covers the main example that we have in mind: that is, the case when $f$ is equal to a constant on some sub-interval in $(a-h_1,a+h_1)\subset (a-h,a+h)$. Note that $C_{\reg}$-regularity outside $(a-h,a+h)$ requires $f$ to be strictly monotone near $a-h$ and $a+h$, which implies that the inclusion must be strict. However, by introducing additional rescaling that depends on $h-h_1$, one can consider examples with arbitrarily small values of $h-h_1$.
\end{rem}
\subsection{Outline of the procedure}
Let us consider the following operator family on $\ell^2(\Z^d)$:
\bee
\label{eq_h_schr_def}
(H(x)\psi)_{\bn}=\varepsilon(\Delta\psi)_{\bn}+f(x+\bn\cdot\omega)\psi_{\bn},
\ene
where $x\in \R\setminus(\Z+1/2+\Z^d\cdot\omega)$, and $f$ is extended into $\R\setminus(\Z+1/2)$ by $1$-periodicity. Operator \eqref{eq_h_schr_def} does not satisfy the assumptions of Theorem \ref{main}, as $f$ is not strictly monotone. However, it is possible to replace it by a unitarily equivalent operator with non-constant hopping terms which will satisfy the criteria described in Section 6. We will describe the corresponding unitary transformations in a general form. Suppose that
$$
(H\psi)_{\bm}=V_{\bm}\psi_{\bm}+\ep\sum\limits_{\bm'\in \mathbb Z^d}\Phi_{\bm\bm'}\psi_{\bm'}.
$$
To begin with, we will define the basic ``building block'' that will be used in the construction of the required unitary transformation. Let $\bn\neq \bze$. Suppose that $V_{\bze}=H_{\bze\bze}=0$, $V_{\bn}=H_{\bn\bn}>0$, and 
$$
H_{\bze\bn}=\overline{H_{\bn\bze}}=\varepsilon\Phi_{\bze\bn}=\ep\overline{\Phi_{\bn\bze}}.
$$ 
Put
$$
D:=\sqrt{V_{\bn}^2+\ep^2|\Phi_{\bze\bn}|^2}.
$$
Consider the operator $U$ with the following matrix elements:
$$
U_{\bze\bze}=U_{\bn\bn}=\frac{V_{\bn}}{D};
$$
$$
U_{\bze\bn}=\ep\frac{\overline{\Phi_{\bze\bn}}}{D},\quad U_{\bn\bze}=-\ep\frac{\Phi_{\bn\bze}}{D}=-\ep\frac{\overline{\Phi_{\bze\bn}}}{D};
$$
$$
U_{\bm\bm}=1\text{ for }\bm\in\mathbb Z^d\setminus\{\bze,\bn\};
$$
and all other matrix elements of $U$ are zero. One can check that $U$ is a unitary operator on $\ell^2(\Z^d)$. Let us calculate the conjugation of $H$ by $U$:
$$
U^* H U=\widetilde V+\ep \widetilde \Phi,
$$
where
$$
\widetilde V_{\bze}=-\ep^2\frac{V_{\bn}|\Phi_{\bze\bn}|^2}{D^2},\quad \widetilde V_{\bn}=\frac{V_{\bn}^3+2\ep^2 |\Phi_{\bze\bn}|^2V_{\bn}}{D^2},\quad \widetilde V_{\bm}=V_{\bm} \text{ for }\bm\in \Z^d\setminus\{\bze,\bn\};
$$
$$
\widetilde\Phi_{\bze\bn}=-\ep^2\frac{|\Phi_{\bze\bn}|^2\Phi_{\bn\bze}}{D^2},\quad \widetilde\Phi_{\bn\bze}=-\ep^2\frac{|\Phi_{\bze\bn}|^2\Phi_{\bze\bn}}{D^2},
$$
$$
\widetilde \Phi_{\bze\bm}=\frac{V_{\bn}}{D}\Phi_{\bze\bm}+\ep \frac{\Phi_{\bn\bze}}{D}\Phi_{\bn\bm},\quad \widetilde \Phi_{\bn\bm}=\frac{V_{\bn}}{D}\Phi_{\bn\bm}-\ep \frac{\Phi_{\bze\bn}}{D}\Phi_{\bze\bm}\, \text{ for }\,\bm\in\Z^d\setminus\{\bze,\bn\},
$$
and $\widetilde \Phi_{\bm_1\bm_2}=\Phi_{\bm_1\bm_2}$ for $\bm_1,\bm_2\in\Z^d\setminus\{\bze,\bn\}$. 
\begin{remark}
Clearly, we can apply this transform even if $V_{\bze}\ne 0$, as long as $V_{\bn}\ne V_{\bze}$: we represent $H$ as $H=V_{\bze}I+\hat H$ and transform $\hat H$ using the procedure described above. Similarly, one can apply the procedure to any pair of lattice points $\{\bk,\bk+\bn\}$ through adding $\bk$ to all indices in the above procedure: in particular, $D=\sqrt{(V_{\bn+\bk}-V_{\bk})^2+\ep^2|\Phi_{\bk,\bn+\bk}|^2}$.
\end{remark}

This procedure will be called {\it elimination of the entry $\Phi_{\bze\bn}$}. Note that, strictly speaking, it only eliminates it up to the order $\ep^2$. As a result, new entries can also be created, such as $\widetilde \Phi_{\bm\bn}$ in the above calculations. Using translation, one can describe elimination of any entry $\Phi_{\bk,\bk+\bn}$ under similar assumptions.

It is natural to consider the elimination procedure in the context of long range operators from Section 7.1. For example, if one starts from an operator $H=V+\ep \Phi^1$, the result of the transformation can be represented as
$$
U^* H U=\widetilde V+\ep \widetilde \Phi^1+\ep^2\widetilde \Phi^2+\ep^3 \widetilde \Phi^3.
$$
Here, we will ignore the dependence of $D$ on $\varepsilon$ and group the correction terms based on the number of $\ep$ factors in front of them: for example, in the earlier notation, $\widetilde{\Phi}^3_{\bn\bze}=-\frac{|\Phi_{\bze\bn}|^2\Phi_{\bze\bn}}{D^2}$. In application, the dependence on $\ep$ through $D$ will be slow enough so that the upper required bounds on the off-diagonal terms and their derivatives will be uniform in $\ep$. On the other hand, one will need to deal more carefully with the diagonal terms, since the presence of the $\ep^2$ terms will be crucial for establishing monotonicity.

 In this new notation, one can state that the procedure completely eliminates $\Phi^1_{\bze\bn}$, since the correction to this term goes into $\widetilde \Phi^{3}$. One can repeat this procedure, say, for the pair of lattice points $\{\bze,\bm\}$. In this case, we will use the operator $U$ constructed from $\widetilde V+\ep \widetilde \Phi^1$, without including further terms. As a result, we will eliminate both $\Phi^1_{\bze\bm}$ and $\Phi^1_{\bze\bn}$. This may introduce some new higher order terms. The exact result of the operations will, in particular, depend on the order in which $\Phi^1_{\bze\bm}$ and $\Phi^1_{\bze\bn}$ are eliminated, but that dependence will only affect higher order terms.

We intend to apply a sequence of transformations of the above kind to the operator \eqref{eq_h_schr_def}, where $f$ satisfies the bullet points in the beginning of the section. The goal of these transformations will be to construct a new operator, unitarily equivalent to $H(x)$, satisfying the assumptions made in Section 6.2. We will start by considering $H(x)$ as a ``long range'' operator $H_0(x)=V_0(x)+\ep \Phi^1_0(x)$, where $\Phi_0^1(x)=\Delta$. During later steps, the operator will have non-constant off-diagonal terms of higher range than $\Delta$. In the notation $\Phi_k^j(x)_{\bm\bn}$ above and below, $k$ is the step of the procedure, $j$ is the order of the perturbation, and $\bm\bn$ indicate the matrix element under consideration. The transformation procedure will be performed in two steps as follows.
\begin{itemize}
	\item For every $\bm\in \Z^d$ such that $x+\bm\cdot\omega\in [a-h,a+h]$, eliminate all entries $\Phi_0^1(x)_{\bm\bn}$. This procedure can be performed in a covariant way, so that the resulting operator family $H_1(x)$ will still be quasiperiodic. After this first step, the family $H_1(x)=V_1(x)+\ep\Phi_1^1(x)+\ep^2\Phi_1^2(x)+\ep^3\Phi_1^3(x)$ will satisfy two additional properties: the diagonal entries will have lower bounds on the derivatives of the order $\ep^2$, and for every $\bm\in \Z^d$ such that $x+\bm\cdot\omega\in [a-h,a+h]$, all edges starting at $\bm$ will have weight at least two (in other words, $\Phi_1^1(x)_{\bm\bn}=0$).
	\item On the second step, eliminate all entries $\Phi_1^{2}(x)_{\bm\bn}$ for all $\bm\in \Z^d$ with $x+\bm\cdot\omega\in [a-h,a+h]$ (again, one needs to do it in a covariant way described below in more detail).
\end{itemize}
Afterwards, one needs to check that the new operator family $H_2(x)$ satisfies the assumptions (sing0) -- (sing4) in Subsection 7.2.
\subsection{Step 1}For $x\in [a-h,a+h]$ and $\bm\in\Z^d$ with $|\bm|_1=1$, denote by $U_{1,\bm}(x)$ the unitary operator that eliminates the off-diagonal entry $\Phi^1_0(x)_{\bm\bn}=(\Delta)_{\bze\bm}(x)=1$ from $H(x)$. Let
$$
U_1(x)=\prod_{\bm\in \Z^d\colon |\bm|_1=1}U_{1,\bm}(x).
$$
Let $U_2(x)=U_1(x)$ for $x\in [a-h,a+h]$, $U_2(x)=I$ for $x\in [-1/2,1/2)\setminus[a-2h,a+2h]$, and ``unitarily interpolated'' in between:
$$
U_2(a-2h+th)=\left((1-t)I+tU_1(a-h)\right)\left|(1-t)I+tU_1(a-h)\right|^{-1},\quad t\in [0,1],
$$
and similarly on $[a+h,a+2h]$. Additionally, extend $U_2(x)$ from $[-1/2,1/2)$ to $\mathbb R$ by $1$-periodicity and continuity in the variable $x$.

Next, we would like to spread this transformation to all lattice points by considering
$$
U_3(x)=\prod\limits_{\bn\in \mathbb Z^d}T_{\bn}U_2(x+\bn\cdot\omega)T_{-\bn},
$$
where $T_{\bn}$ denotes the lattice translation operator. Note that the factors in the product commute and the product is well-defined in the strong operator topology: for any vector $\psi\in \ell^2(\Z^d)$ with finite support, all factors of $U_3(x)$, except for finitely many, will act on $\psi$ as an identity operator, and the remaining finitely many will commute with each other.

Define
$$
H_1(x):=U_3(x)^*H(x)U_3(x)=V_1(x)+\ep \Phi_1^1(x)+\ep^2 \Phi_1^2(x)+\ep^3\Phi_1^3(x).
$$
Here, we will use several assumptions discussed above: $\Phi_1^j(x)$ may depend on $\ep$ through the $D$ factors. We will also absorb all higher order terms into $\Phi_1^3(x)$. Let $f_1(x):=V_1(x)_{\bze\bze}$.
\begin{lem}
\label{lemma_step_one} Suppose that $H(x)$ satisfies the assumptions from Subsection $7.1$. Construct $U_3(x)$ and $H_1(x)$ as above. There exists $\ep_0=\ep_0(f,\omega)>0$ and $c_1=c_1(f,\omega)>0$, $c_2(f,\omega)>0$, $c_3(d)>0$ such that, for $0<\ep<\ep_0$ we have the following:
\begin{enumerate}
	\item $f_1$ is $(C_{\reg}-c_2\ep)$-regular on $[-1/2,a-h]\cup[a+h,1/2]$.
	\item $f_1'(x)\ge 1-c_2\ep$ for $x\in [-1/2,a-h]\cup[a+h,1/2]$.
	\item $f_1'(x)\le D+c_2\ep$ for $x\in [a-2h,a+2h]$.
	\item $\|(\Phi^j_1)_{\bm\bn}\|_{C^1}\le c_2$.
	\item The range of $\Phi_1^2(x)$ is at most $2$; in other words, $\Phi^2_1(x)_{\bm\bn}\neq 0$ can only happen for $|\bm-\bn|_1\le 2$. The range of $\Phi^3_1(x)$ is bounded by $c_3(d)$.
	\item For $x+\bm\cdot\omega\in [a-h,a+h]+\Z$, we have $\Phi^1_1(x)_{\bm\bn}=0$.
	\item For $x\in [a-h,a+h]+\Z$, we have $f_1'(x)\ge c_1 \ep^2$.
	\end{enumerate}
\end{lem}
\begin{proof}
Properties (1) -- (4) follow from the fact that all matrices $U_j(x)$ involved in the process have finite range are small perturbations of identity with $\|(I-U_j(x))_{\bm\bn}\|_{C^1}\le c(\omega,f) \ep$. The latter follows from the fact that the entries of $U_j(x)$ are obtained from $\frac{1}{f(x)-f(x+\bn\cdot\omega)}$, where $\dist(\bn\cdot\omega,\Z)\ge 2h$ (and therefore $f$ is regular either at $x$ or at $x+\bn\cdot\omega$). Property (5) follows from the construction of the operators $U$: one can note that any edge created between $\bn$ and $\bm$ will have length at least $|\bm-\bn|$. Note that, since we absorb all lower order terms into $\Phi_1^3$, we cannot argue that its range is bounded by $3$. However, it is clearly bounded by some constant that only depends on the number of matrix multiplications which, in turn, only depends on the dimension. Property (6) directly follows from the construction of $U$. Property (7) is the most important property which follows from the calculations of Section 7.2. For $x\in [a-h,a+h]$, we have
$$
f_1(x)=f(x)+\varepsilon^2 \sum_{j=1}^d \left(\frac{1}{f(x)-f(x+\omega_j)}+\frac{1}{f(x)-f(x-\omega_j)}\right)+O(\ep^3),
$$
where the last term admits a $C^1$-uniform upper bound in $x$ on $[a-h,a+h]$. In particular, we have $f_1'(x)\ge c\ep^2$ on $[a-h,a+h]$, where $c=c(f,\omega)>0$; note that both $x+\omega_j$ and $x-\omega_j$ are outside of $[a-h,a+h]\,\mathrm{mod}\,\Z$, and therefore the derivatives of $f$ are bounded from below in that region.
\end{proof}
\subsection{Step 2} The operator $H_1(x)$ is a ``borderline'' case which just misses the assumptions of (sing4) and thus Theorem \ref{th_main_singular}. Indeed, we have only one singular interval with $\mu_1=2$, $\nu_1=1$. The shortest possible loop starting from a singular small denominator has length $4$ (two steps of length two). By attaching multiple copies of this loop, we can have $\singdownedges(\CP)\sim |\CP|/4$, which fall short of satisfying (sing4).

The above issue can be resolved by performing another step of eliminating diagonal entries. For each $\bn\in \Z^d$ with $|\bn|_1\le 2$ and $x\in [a-h,a+h]$, let $U_{3,\bn}(x)$ be the unitary operator which eliminates the entry $\Phi_1^2(x)_{\bze\bn}$ in the operator $V_1(x)+\ep^2 \Phi_1^2(x)$. Note that, due to (6) in Lemma \ref{lemma_step_one}, we already have $\Phi_1^1(x)_{\bze\bn}$ eliminated for $x$ under consideration. Let $U_3(x)$ be the operator obtained from $U_{3,\bn}(x)$ in the same way as $U_3$ was obtained from $U_1$ (that is, interpolation with the identity and extension by covariance).
\begin{lem}
\label{lemma_step_two}
Under the assumptions of Lemma $\ref{lemma_step_one}$ and with the above construction of $U_3(x)$, let 
$$
H_2(x):=U_3(x)^*H_1(x)U_3(x)=V_2(x)+\ep \Phi_2^1(x)+\ep^2 \Phi_2^2(x)+\ep^3\Phi_2^3(x).
$$ Then, the operator $H_2=V_2+\ep \Phi^1_2+\ep^2 \Phi^2_2+\ep^3\Phi^3_2$ satisfies $(1)$--$(7)$ from Lemma $\ref{lemma_step_one}$. In addition, $\Phi_2^2(x)_{\bm\bn}=0$ for $x+\bm\cdot\omega\in [a-h,a+h]+\Z$.
\end{lem}
\begin{proof}
The conclusions of Lemma \ref{lemma_step_one} remain true under perturbations of the diagonal entries by $O(\ep^3)$ and off-diagonal entries of $O(\ep^2)$, which is clearly the case if one conjugates $H_1$ using $U_3$. The additional claim follows from the construction of $U_3$, since it completely eliminates $\Phi_1^2(x)_{\bm\bn}$ with $x+\bm\cdot\omega\in [a-h,a+h]+\Z$, possibly creating additional terms of order $\ep^3$ or higher.
\end{proof}
\begin{thm}
\label{th_main_flat}
Let $H(x)$ be an operator of the class defined in Section $7.1$. Construct the operator $H_2(x)$ using $U_1(x)$ -- $U_3(x)$ as above, and consider the loop expansion for $H_2(x)$ as in Section $7.1$. There exist $\beta_{\max}=\beta_{\max}(f,\omega)>0$, $C(f,\omega)>0$ and $\ep_0=\ep_0(f,\omega)>0$ such that, for
$$
0<\ep<\ep_0,\quad 0<\beta<\beta_{\max},
$$
the operator $H_2(x)$ satisfies $(\mathrm{sing4})$ from Section $7.3$ and, as a consequence, the perturbation series for $H_2(x)$ convergence and both $H(x)$ and $H_2(x)$ satisfy Anderson localization for all $x\in \R\setminus(\Z+1/2+\Z^d\cdot\omega)$.
\end{thm}
\begin{proof}
In order to check (sing4) for $H_2(x)$, let $\CP$ be a loop stack. For a non-base loop $\CL$ on $\CP$, let us consider the total number of $\delta$ factors produced by this loop (including the factor from its attachment). The last bullet point in the beginning of Section 7.1 implies that, starting from the interval $[a-h,a+h]$, the total length of steps to return to that interval without ending the loop will be at least $7$. As a consequence, {\it any loop $\CL$ with $|\CL|\le 6$ does not have any singular small denominators on it}. Therefore, the contribution of this loop is of the order $\delta^{-2}\sim\ep^{-4}$, which gives a total contribution of at most $O(\ep^{-2k_1/3})$, where $k_1$ is the total length of these loops.

Let $k_2$ be the total length of loops $\CL$ on $\CP$ with $|\CL|\ge 7$ (note that $H_2(x)$ may contain off-diagonal entries of odd ranges). The total contribution of the attachment factors from these loops is $\ep^{-4 k_2/7}$. Additionally, each such loop can have at most $\lfloor |\CL|/6\rfloor$ singular small denominators (as the result of Lemma \ref{lemma_step_two}). As a consequence, the total contribution from $\delta$-factors from these loops will be $O(\ep^{-4k_2/7-2k_2/6})=O(\ep^{-19k_2/21})$. By combining all bounds, we arrive to (conv4) with $r=19/21$.
\end{proof}
\begin{rem}
Suppose that $f$ indeed has a flat piece. One can check that then the integrated density of states of our operator $H$ has a lower derivative number of order $\sim\varepsilon^{-2}$ on an interval of length  $\sim\varepsilon^{2}$; outside this interval, derivative numbers are bounded. Thus, if the `non-integrated density of states' exists, it has a pike of order $\sim\varepsilon^{-2}$ on an interval of length  $\sim\varepsilon^{2}$ and is bounded otherwise. 
\end{rem}
\begin{rem}
The assumptions on $f$ listed at the beginning of this section are not optimal. First of all, we can assume, for example, that $f$ is non-decreasing when  $x\in [a-h,a+h]$ and $f$ is Lipschitz monotone (meaning $f'(x)>1$) outside of this interval. This allows $f$ to have continuous derivative (or even be infinitely smooth). 
The assumptions on the `number of steps' we need to make to return from the flat piece back to the flat piece  
can also be improved. We plan to come back to this class of examples in a subsequent publication and discuss other effects occurring  here. 
\end{rem}
\end{document}